\documentclass[11pt,a4paper]{article}
\usepackage{amssymb}
\usepackage{amsmath}
\usepackage{latexsym}
\usepackage{amsthm}
\usepackage{dsfont}
\usepackage{mathtools}
\usepackage{cellspace}
\usepackage{tikz}

\newtheorem{thm}{Theorem}
\newtheorem{lem}[thm]{Lemma}

\usepackage{xpatch}
\xpatchcmd{\proof}{\itshape}{\normalfont\proofnameformat}{}{}
\newcommand{\proofnameformat}{}

\begin{document}

\renewcommand{\proofnameformat}{\bfseries}

\begin{center}
{\Large\textbf{Eigenvalues of random matrices from compact classical groups in Wasserstein metric}}

\vspace{10mm}

\textbf{Bence Borda}

{\footnotesize Graz University of Technology

Steyrergasse 30, 8010 Graz, Austria

Email: \texttt{borda@math.tugraz.at}}

\vspace{5mm}

{\footnotesize \textbf{Keywords:} circular unitary ensemble, spectral measure, characteristic polynomial, \\ determinantal point process, optimal transport}

\end{center}

\vspace{4mm}

\begin{abstract}
The circular unitary ensemble and its generalizations concern a random matrix from a compact classical group $\mathrm{U}(N)$, $\mathrm{SU}(N)$, $\mathrm{O}(N)$, $\mathrm{SO}(N)$ or $\mathrm{USp}(N)$ distributed according to the Haar measure. The eigenvalues are known to be very evenly distributed on the unit circle. In this paper, we study the distance from the empirical measure of the eigenvalues to uniformity in the quadratic Wasserstein metric $W_2$. After finding the exact value of the expected value and the variance, we deduce a limit law for the square of the Wasserstein distance. We reformulate our results in terms of the $L^2$ average of the number of eigenvalues in circular arcs, and also in terms of the characteristic polynomial of the matrix on the unit circle.
\end{abstract}

\section{Introduction}

A central problem in random matrix theory concerns the empirical spectral measure of random matrix models. Working with large matrices with i.i.d.\ entries, under very general assumptions the suitably scaled empirical measure of the eigenvalues converges either to Wigner's semicircle law (in the case of Hermitian matrices) or to the uniform measure on the complex unit disk (in the case of non-Hermitian matrices) as the dimension of the matrices increases. Estimating the rate of convergence in Wasserstein metric has attracted attention in recent years \cite{BGM,CHM,DA,JA,MM3,OW}.

In this paper, we work with Dyson's circular unitary ensemble \cite{DY} and some of its generalizations. That is, we consider a matrix drawn uniformly at random from a compact classical group: $\mathrm{U}(N)$, $\mathrm{SU}(N)$, $\mathrm{O}(N)$, $\mathrm{SO}(N)$ or $\mathrm{USp}(N)$. The $N$ eigenvalues then lie on the unit circle, and tend to be very evenly distributed. We refer to Diaconis \cite{DI1}, Meckes \cite{M} and Mehta \cite{ME} for a survey and applications.

The eigenvalues and the characteristic polynomial of a random unitary matrix have been extensively used to model the zeros and the value distribution of the Riemann zeta function since the seminal paper of Montgomery \cite{MO}. The low-lying zeros of certain families of $L$-functions are similarly modeled, depending on the symmetry type of the family, by the eigenvalues of a random orthogonal or symplectic matrix, further motivating the study of random matrices from compact classical groups. We refer to Katz and Sarnak \cite{KS}, and Keating and Snaith \cite{KSn2} for a survey.

To state our results, we introduce the following notation. Let $\mathbb{T} = \{ z \in \mathbb{C} \, : \, |z|=1 \}$ denote the unit circle, and let $\lambda_{\mathbb{T}}$ be the normalized Haar measure on $\mathbb{T}$. The metric on $\mathbb{T}$ is the arc length metric $d(z,w)= \min \{ \mathrm{arg} (z \bar{w}), 2 \pi - \mathrm{arg} (z \bar{w}) \}$, where $\bar{w}$ is the complex conjugate of $w$, and $\mathrm{arg}$ is the argument in $[0,2 \pi)$. Given $1 \le p < \infty$, the Wasserstein metric $W_p$ is a metric on the set of Borel probability measures on $\mathbb{T}$ defined as
\[ W_p (\mu, \nu) = \inf_{\vartheta \in \mathrm{Coupling} (\mu, \nu)} \left( \int_{\mathbb{T}^2} d(z,w)^p \, \mathrm{d}\vartheta (z,w) \right)^{1/p} , \]
where $\mathrm{Coupling} (\mu, \nu)$ is the set of Borel probability measures on $\mathbb{T}^2$ with marginals $\mu$ and $\nu$. The Wasserstein metric was originally introduced in the theory of optimal transport, and has applications in many different fields including geometric analysis, fluid dynamics and computer science. The case $p=2$ is particularly well studied because of its connections to entropy and logarithmic Sobolev inequalities \cite{OV}. We refer to Villani \cite{VI} for an introduction and further context.

We use the notation $\mathrm{O}^- (N) = \{ A \in \mathrm{O}(N) \, : \, \det A = -1 \}$ for the coset of $\mathrm{SO}(N)$ in $\mathrm{O}(N)$. By the normalized Haar measure on $\mathrm{O}^- (N)$ we mean two times the restriction of the normalized Haar measure on $\mathrm{O}(N)$.

Let $G=\mathrm{U}(N)$, $\mathrm{SU}(N)$, $\mathrm{SO}(2N+1)$, $\mathrm{O}(2N+1)$, $\mathrm{SO}(2N)$, $\mathrm{O}^- (2N+2)$ or $\mathrm{USp}(2N)$ with some integer $N \ge 1$, and let $A \in G$ be a random matrix distributed according to the normalized Haar measure on $G$. Note that $1$ and/or $-1$ must be an eigenvalue of $A$ for certain $G$ for parity reasons. See Table \ref{tablecGsigmaG} for the set of these trivial eigenvalues. We call all other eigenvalues of $A$ nontrivial, and denote their number by $N_0$. Let
\[ \mu_A = \frac{1}{N_0} \sum_{\lambda \textrm{ nontrivial eigenvalue of } A} \delta_{\lambda} \]
denote the empirical measure of the nontrivial eigenvalues of $A$.

We are interested in $W_p (\mu_A, \lambda_{\mathbb{T}})$, the distance to uniformity in Wasserstein metric. Meckes and Meckes \cite{MM1,MM2} showed that for any $1 \le p < \infty$,
\[ \mathbb{E}\, W_p (\mu_A, \lambda_{\mathbb{T}}) \ll p \frac{\sqrt{\log N}}{N} \]
with a universal implied constant, and proved a large deviation inequality for $W_p (\mu_A, \lambda_{\mathbb{T}})$. In this paper, we only work with the quadratic Wasserstein metric $W_2$. Our first result establishes the precise asymptotics of the expected value and the variance of $W_2^2 (\mu_A, \lambda_{\mathbb{T}})$.
\begin{thm}\label{maintheorem} Let $G=\mathrm{U}(N)$, $\mathrm{SU}(N)$, $\mathrm{SO}(2N+1)$, $\mathrm{O}(2N+1)$, $\mathrm{SO}(2N)$, $\mathrm{O}^- (2N+2)$ or $\mathrm{USp}(2N)$, and let $A \in G$ be a random matrix distributed according to the normalized Haar measure on $G$. We have
\[ \begin{split} \mathbb{E} \, W_2^2 (\mu_A, \lambda_{\mathbb{T}}) &= 2 \frac{\log N_0}{N_0^2} + \frac{c_G}{N_0^2} + O \left( \frac{1}{N^3} \right) , \\ \mathrm{Var} \, W_2^2 (\mu_A, \lambda_{\mathbb{T}}) &= \frac{\sigma_G}{N_0^4} + O \left( \frac{1}{N^5} \right) \end{split} \]
with universal implied constants. The number of nontrivial eigenvalues $N_0$ and the constants $c_G$ and $\sigma_G$ are given in Table \ref{tablecGsigmaG}.
\end{thm}

\begin{table}[b]
\centering
\begin{tabular}{|Sc|Sc|Sc|Sc|Sc|}
\hline
$G$ & Trivial eigenvalues & $N_0$ & $c_G$ & $\sigma_G$ \\
\hline \hline
$\mathrm{U}(N)$, $\mathrm{SU}(N)$ & $\emptyset$ & $N$ & $2 \gamma +2$ & $\displaystyle{\frac{2\pi^2}{3}}$ \\
\hline
$\mathrm{SO}(2N+1)$, $\mathrm{O}(2N+1)$ & $\{ \det A \}$ & $2N$ & $\displaystyle{2 \gamma +2 + \frac{\pi^2}{4}}$ & $\displaystyle{\frac{4\pi^2}{3} + 14 \zeta (3)}$ \\
\hline
$\mathrm{SO}(2N)$, $\mathrm{USp}(2N)$ & $\emptyset$ & $2N$ & $\displaystyle{2 \gamma +2+\frac{\pi^2}{12}}$ & $\displaystyle{\frac{4 \pi^2}{3} + 2 \zeta (3)}$ \\
\hline
$\mathrm{O}^- (2N+2)$ & $\{ 1, -1 \}$ & $2N$ & $\displaystyle{2 \gamma +2 + \frac{\pi^2}{12}}$ & $\displaystyle{\frac{4 \pi^2}{3} + 2 \zeta (3)}$ \\
\hline
\end{tabular}
\caption{The set of trivial eigenvalues, the number of nontrivial eigenvalues $N_0$, and the constants $c_G$ and $\sigma_G$. Note that $\gamma$ is the Euler--Mascheroni constant, and $\zeta$ is the Riemann zeta function.}
\label{tablecGsigmaG}
\end{table}

We actually find exact formulas for $\mathbb{E} \, W_2^2 (\mu_A, \lambda_{\mathbb{T}})$ and $\mathrm{Var} \, W_2^2 (\mu_A, \lambda_{\mathbb{T}})$ without an error term, see Theorem \ref{exacttheorem}. Our second result is a limit law for $W_2^2 (\mu_A, \lambda_{\mathbb{T}})$ after natural centering and scaling.

\begin{thm}\label{limitlawtheorem} Let $G=\mathrm{U}(N)$, $\mathrm{SU}(N)$, $\mathrm{SO}(2N+1)$, $\mathrm{O}(2N+1)$, $\mathrm{SO}(2N)$, $\mathrm{O}^- (2N+2)$ or $\mathrm{USp}(2N)$, and let $A \in G$ be a random matrix distributed according to the normalized Haar measure on $G$. We have
\[ N_0^2 W_2^2 (\mu_A, \lambda_{\mathbb{T}}) - 2 \log N_0 - c_G \overset{d}{\to} \xi_G \qquad \textrm{as } N \to \infty . \]
The random variable $\xi_G$ is defined as an infinite series of independent random variables in Table \ref{tablexiG}, and its characteristic function is given in Table \ref{tablexiGchar}.
\end{thm}

\begin{table}[h]
\centering
\begin{tabular}{|Sc|Sc|}
\hline
$G$ & Series representation of $\xi_G$ \\
\hline \hline
$\mathrm{U}(N)$, $\mathrm{SU}(N)$ & $\displaystyle{\sum_{k=1}^{\infty} \frac{X_k^2 + Y_k^2 -2}{k}}$ \\
\hline
$\mathrm{SO}(2N+1)$, $\mathrm{O}(2N+1)$ & $\displaystyle{2 \sum_{k=1}^{\infty} \frac{X_k^2 - \mathds{1}_{\{ k \textrm{ odd}\}} (2/\sqrt{k}) X_k -1}{k}}$ \\
\hline
$\begin{array}{cc} \mathrm{SO} (2N), \mathrm{O}^- (2N+2), \\ \mathrm{USp}(2N) \end{array}$ & $\displaystyle{2 \sum_{k=1}^{\infty} \frac{X_k^2 - \mathds{1}_{\{ k \textrm{ even}\}} (2/\sqrt{k}) X_k -1}{k}}$ \\
\hline
\end{tabular}
\caption{The definition of $\xi_G$ in terms of i.i.d.\ standard Gaussians $X_k, Y_k$, $k \ge 1$.}
\label{tablexiG}
\end{table}

\begin{table}[h]
\centering
\begin{tabular}{|Sc|Sc|}
\hline
$G$ & Characteristic function of $\xi_G$ \\
\hline \hline
$\mathrm{U}(N)$, $\mathrm{SU}(N)$ & $\displaystyle{\Gamma (1-2it) e^{-2\gamma it}}$ \\
\hline
$\mathrm{SO}(2N+1)$, $\mathrm{O}(2N+1)$ & $\displaystyle{\Gamma (1-4it)^{1/2} \exp \left( - \left( 2 \gamma + \frac{\pi^2}{4} \right) it - \frac{2\psi (1-4it) - \psi (1-2it)+\gamma}{4} \right)}$ \\
\hline
$\begin{array}{cc} \mathrm{SO} (2N), \mathrm{O}^- (2N+2), \\ \mathrm{USp}(2N) \end{array}$ & $\displaystyle{\Gamma (1-4it)^{1/2} \exp \left( - \left( 2 \gamma + \frac{\pi^2}{12} \right) it - \frac{\psi (1-2it)+\gamma}{4} \right)}$ \\
\hline
\end{tabular}
\caption{The characteristic function of $\xi_G$ in terms of the gamma function $\Gamma (z)$ and the digamma function $\psi (z)=\Gamma'(z)/\Gamma (z)$. We use the unique complex square root that makes the function $\Gamma (1-4it)^{1/2}$, $t \in \mathbb{R}$ continuous with value $1$ at $t=0$.}
\label{tablexiGchar}
\end{table}

We show in the Appendix that $\mathbb{E} \, \xi_G =0$ and $\mathrm{Var} \, \xi_G = \sigma_G$ for all $G$, and that the density function of $\xi_{\mathrm{U}(N)} = \xi_{\mathrm{SU}(N)}$ is
\[ \frac{1}{2} \exp \left( - \frac{x+2 \gamma}{2} - \exp \left( - \frac{x+2 \gamma}{2} \right) \right), \qquad x \in \mathbb{R} . \]
The density function of $\xi_G$ for the other groups seem not to be elementary functions, however.

Our approach relies on two crucial facts. First, we use the fact that the nontrivial eigenvalues form a deterministic point process, see Section \ref{determinantalsection}. The second crucial fact is that the distance from an arbitrary Borel probability measure $\mu$ on $\mathbb{T}$ to uniformity in Wasserstein metric has the exact formula
\begin{equation}\label{wassersteinexact}
W_2^2 (\mu, \lambda_{\mathbb{T}}) = \sum_{k \in \mathbb{Z} \backslash \{ 0 \}} \frac{|\hat{\mu}(k)|^2}{k^2},
\end{equation}
where $\hat{\mu}(k) = \int_{\mathbb{T}} z^k \, \mathrm{d}\mu (z)$ are the Fourier coefficients of $\mu$. We refer to Graham \cite{GR} for a proof of \eqref{wassersteinexact} based on the explicit solution of the optimal transport problem on $\mathbb{T}$ and the Parseval formula. Unfortunately, there is no similar exact formula for $W_p$, $p \neq 2$, or on higher dimensional spaces.

Theorems \ref{maintheorem} and \ref{limitlawtheorem} are closely connected to some previous results on the empirical spectral measure of random matrices. In the case of $G=\mathrm{U}(N)$, Rains \cite{RA} showed that the normalized number of eigenvalues that fall in a fixed circular arc $J \subseteq \mathbb{T}$ has expected value $\mathbb{E} \, \mu_A (J) = \lambda_{\mathbb{T}}(J)$ and variance
\begin{equation}\label{rains}
\mathrm{Var} \, \mu_A (J) = \frac{\log N}{\pi^2 N^2} + \frac{\gamma +1+\log |2 \sin (\pi \lambda_{\mathbb{T}}(J))|}{\pi^2 N^2} + O \left( \frac{1}{N^3} \right) .
\end{equation}
The centered and scaled sequence $(\mu_A (J) - \lambda_{\mathbb{T}}(J)) \pi N /\sqrt{\log N}$ converges to a standard Gaussian \cite{WI}. A similar central limit theorem holds in other classical compact groups \cite[p.\ 104]{M}, and also in other random matrix models \cite{CL}. We refer to \cite{DI2,DI3,JO,SO} for further results on linear statistics of the eigenvalues.

Now let $J(z,w)$ denote the circular arc connecting $z \in \mathbb{T}$ to $w \in \mathbb{T}$ counterclockwise. It turns out that for an arbitrary Borel probability measure $\mu$, the $L^2$ average over all circular arcs,
\[ \Delta_2 (\mu, \lambda_{\mathbb{T}}) \coloneqq \left( \int_{\mathbb{T}} \int_{\mathbb{T}} \left| \mu (J(z,w)) - \lambda_{\mathbb{T}}(J(z,w)) \right|^2 \, \mathrm{d} \lambda_{\mathbb{T}} (z) \, \mathrm{d} \lambda_{\mathbb{T}}(w) \right)^{1/2} \]
can be expressed in terms of the Fourier coefficients as
\[ \Delta_2 (\mu, \lambda_{\mathbb{T}})^2 = \frac{1}{2 \pi^2} \sum_{k \in \mathbb{Z} \backslash \{ 0 \}} \frac{|\hat{\mu}(k)|^2}{k^2} . \]
This is a commonly used measure of equidistribution known as periodic $L^2$ discrepancy or diaphony \cite{BGM,LE}. In particular, $W_2 (\mu, \lambda_{\mathbb{T}}) = \sqrt{2} \pi \Delta_2(\mu, \lambda_{\mathbb{T}})$. Our results can thus be restated in terms of $\Delta_2 (\mu_A, \lambda_{\mathbb{T}})$, the $L^2$ average of the eigenvalue count in circular arcs. In the case $G=\mathrm{U}(N)$, Bao and He \cite{BH} proved the limit law
\[ 2 \pi^2 N^2 \Delta_2 (\mu_A, \lambda_{\mathbb{T}})^2 - 2\log N -2 \gamma -2 \overset{d}{\to} \xi_{\mathrm{U}(N)} , \]
although they did not connect it to the Wasserstein metric. Our Theorem \ref{limitlawtheorem} for $G=\mathrm{U}(N)$ thus follows from their result.

Theorems \ref{maintheorem} and \ref{limitlawtheorem} can also be reformulated in terms of the characteristic polynomial $P_A (z) = \det (zI - A)$. Let
\[ F_A (z) = \sum_{\lambda \textrm{ nontrivial eigenvalue of }A} \log \left( 1-\frac{\lambda}{z} \right) , \qquad z \in \mathbb{T}, \]
defined with the principal branch of the logarithm. Thus
\[ \exp (F_A (z)) = \prod_{\lambda \textrm{ nontrivial eigenvalue of }A} \left( 1 - \frac{\lambda}{z} \right) = \frac{P_A(z)}{z^{N_0} \prod_{\lambda \textrm{ trivial eigenvalue of }A} (z-\lambda)} \]
is a suitably normalized form of the characteristic polynomial so that $F_A (z)$ has mean zero on $\mathbb{T}$. One readily checks that the Fourier coefficients of $F_A (z)$ are
\[ \frac{1}{2 \pi} \int_0^{2 \pi} F_A (e^{ix}) e^{-ikx} \, \mathrm{d} x = \sum_{\lambda \textrm{ nontrivial eigenvalue of }A} \frac{1}{k \lambda^k} = N_0 \frac{\hat{\mu}_A (-k)}{k} , \]
hence by the Parseval formula and \eqref{wassersteinexact},
\[ \frac{1}{2 \pi }\int_0^{2 \pi} |F_A (e^{ix})|^2 \, \mathrm{d} x = N_0^2 \sum_{k \in \mathbb{Z} \backslash \{ 0 \}} \frac{|\mu_A (k)|^2}{k^2} = N_0^2 W_2^2 (\mu_A, \lambda_{\mathbb{T}}) . \]
In particular, our results can be restated in terms of $\| F_A \|_{L^2 (\mathbb{T})} = N_0 W_2 (\mu_A, \lambda_{\mathbb{T}})$.

In comparison, in the case of $G=\mathrm{U}(N)$, Keating and Snaith \cite{KSn} showed that for any fixed $z \in \mathbb{T}$, $F_A (z) / \sqrt{\log N}$ converges to a standard complex Gaussian. The process $F_A(z)$, $z \in \mathbb{T}$ (without any scaling) also satisfies a functional limit theorem in a certain Sobolev space on $\mathbb{T}$ \cite{HKO}. The limit process is a.s.\ not in $L^2 (\mathbb{T})$, however, suggesting that $\| F_A \|_{L^2 (\mathbb{T})}$ diverges. Our results confirm that $\| F_A \|_{L^2 (\mathbb{T})}^2$ indeed diverges at the rate $2 \log N$ with constant size fluctuations.

As for the maximal value of $\log |P_A (z)| = \mathrm{Re}\, F_A (z)$, in the case $G=\mathrm{U}(N)$, Fyodorov and Keating \cite{FK} conjectured that
\[ \max_{z \in \mathbb{T}} \mathrm{Re} \, F_A (z) - \log N + \frac{3}{4} \log \log N \overset{d}{\to} \Xi , \]
where the random variable $\Xi$ has density function $4 e^{-2x} K_0 (2e^{-x})$, $x \in \mathbb{R}$, with $K_{\nu}(x)$ denoting the modified Bessel function of the second kind. The convergence in distribution part of the conjecture was established in \cite{ABB,CMN,PZ1,PZ2}, but the prediction for the limit distribution has not been verified. Similar results for $\max_{z \in \mathbb{T}} \mathrm{Im} \, F_A(z)$ and $\min_{z \in \mathbb{T}} \mathrm{Im} \, F_A(z)$ are closely related to the eigenvalue count in circular arcs, and in particular to
\[ \Delta_{\infty} (\mu_A, \lambda_{\mathbb{T}}) \coloneqq \sup_{z,w \in \mathbb{T}} |\mu_A (J(z,w)) - \lambda_{\mathbb{T}}(J(z,w))|, \]
the distance to uniformity in a circular Kolmogorov metric. A full limit law for $\Delta_{\infty} (\mu_A, \lambda_{\mathbb{T}})$ nevertheless remains open, see also \cite{MM4}.

\section{Preliminaries}

\subsection{A reduction step}

As observed in \cite[Lemma 2.5]{MM1}, the cases $G=\mathrm{U}(N)$ and $G=\mathrm{SU}(N)$ are equivalent, and similarly, the cases $G=\mathrm{O}(2N+1)$ and $G=\mathrm{SO}(2N+1)$ are also equivalent. For the sake of completeness, we include a short proof.
\begin{lem}\label{reductionlemma} Let $1 \le p < \infty$. The distribution of $W_p (\mu_A, \lambda_{\mathbb{T}})$ is the same for $G= \mathrm{U}(N)$ and $G=\mathrm{SU}(N)$. Further, the distribution of $W_p (\mu_A, \lambda_{\mathbb{T}})$ is the same for $G= \mathrm{O}(2N+1)$ and $G=\mathrm{SO}(2N+1)$.
\end{lem}

\begin{proof} Let $\lambda_G$ denote the normalized Haar measure on $G$. A special case of \cite[Lemma 1.5.1]{KS} due to Katz and Sarnak states that for any bounded, Borel measurable function $f: \mathrm{U}(N) \to \mathbb{R}$ that is invariant under multiplication by a scalar unitary matrix, we have $\int_{\mathrm{U}(N)} f \mathrm{d} \lambda_{\mathrm{U}(N)} = \int_{\mathrm{SU}(N)} f \mathrm{d} \lambda_{\mathrm{SU}(N)}$. The function $A \mapsto W_p (\mu_A, \lambda_{\mathbb{T}})$ is clearly invariant, since multiplication by a scalar unitary matrix simply rotates the set of eigenvalues on $\mathbb{T}$, leaving the distance from the rotationally invariant $\lambda_{\mathbb{T}}$ in Wasserstein metric unchanged. Choosing $f(A)=\mathds{1}_{\{ W_p (\mu_A, \lambda_{\mathbb{T}}) \le x \}}$ with an arbitrary $x \ge 0$ thus yields
\[ \lambda_{\mathrm{U}(N)} (\{ A \in \mathrm{U}(N) \, : \, W_p (\mu_A, \lambda_{\mathbb{T}}) \le x \} ) = \lambda_{\mathrm{SU}(N)} \left( \left\{ A \in \mathrm{SU}(N) \, : \, W_p (\mu_A, \lambda_{\mathbb{T}}) \le x \right\} \right) , \]
as claimed.

The map $\mathrm{SO}(2N+1) \to \mathrm{O}^- (2N+1)$, $A \mapsto -A$ preserves the normalized Haar measure. The value of $W_p (\mu_A, \lambda_{\mathbb{T}})$ is clearly invariant under this map, since the nontrivial eigenvalues are simply mapped to their negatives, leaving the distance from $\lambda_{\mathbb{T}}$ in Wasserstein metric unchanged. Therefore for any $x \ge 0$,
\[ \begin{split} \lambda_{\mathrm{SO}(2N+1)} (\{ A \in \mathrm{SO}(2N+1) \, : \, W_p &(\mu_A, \lambda_{\mathbb{T}}) \le x \} ) \\ &= \lambda_{\mathrm{O}^- (2N+1)} (\{ A \in \mathrm{O}^- (2N+1) \, : \, W_p (\mu_A, \lambda_{\mathbb{T}}) \le x \} ) \\ &= \lambda_{\mathrm{O} (2N+1)} (\{ A \in \mathrm{O} (2N+1) \, : \, W_p (\mu_A, \lambda_{\mathbb{T}}) \le x \} ), \end{split} \]
as claimed.
\end{proof}

\subsection{The empirical spectral measure as a determinantal point process}\label{determinantalsection}

Let $G=\mathrm{U}(N), \mathrm{SO}(2N+1), \mathrm{SO}(2N), \mathrm{O}^- (2N+2)$ or $\mathrm{USp}(2N)$, and let $A \in G$ be a random matrix distributed according to the normalized Haar measure on $G$. We write the nontrivial eigenvalues of $A$ in the form given in Table \ref{tabledeterminantal}. In each case, the points $\theta_n \,\, (1 \le n \le N)$ form a determinantal point process on the interval $[a,b]$ with kernel function $K(x,y)$ with respect to the normalized Lebesgue measure on $[a,b]$, where $K(x,y)$ and $[a,b]$ are given in Table \ref{tabledeterminantal}. A detailed proof of this fact was given by Meckes \cite[Chapter 3.2]{M}, see also Katz and Sarnak \cite[Chapter 5]{KS}.

\begin{table}[h]
\centering
\begin{tabular}{|Sc|Sc|Sc|Sc|}
\hline
$G$ & Nontrivial eigenvalues & $K(x,y)$ & $[a,b]$ \\
\hline \hline
$\mathrm{U}(N)$ & $e^{i \theta_n} \,\, (1 \le n \le N)$ & $\displaystyle{\sum_{j=0}^{N-1} e^{i j(x-y)}}$ & $[0,2 \pi]$ \\
\hline
$\mathrm{SO}(2N+1)$ & $e^{\pm i \theta_n} \,\, (1 \le n \le N)$ & $\displaystyle{2 \sum_{j=0}^{N-1} \sin \left( \frac{2j+1}{2} x \right) \sin \left( \frac{2j+1}{2} y \right)}$ & $[0,\pi]$ \\
\hline
$\mathrm{SO}(2N)$ & $e^{\pm i \theta_n} \,\, (1 \le n \le N)$ & $\displaystyle{1+2\sum_{j=1}^{N-1} \cos (j x) \cos (j y)}$ & $[0,\pi]$ \\
\hline
$\mathrm{O}^- (2N+2), \mathrm{USp}(2N)$ & $e^{\pm i \theta_n} \,\, (1 \le n \le N)$ & $\displaystyle{2 \sum_{j=1}^{N} \sin (jx) \sin (jy)}$ & $[0,\pi]$ \\
\hline
\end{tabular}
\caption{The nontrivial eigenvalues, the kernel function $K(x,y)$ and the interval $[a,b]$. The angles are chosen so that $\theta_n \in [a,b]$.}
\label{tabledeterminantal}
\end{table}

We refer to \cite{HKPV} for a general introduction to determinantal point processes. For our purposes, it suffices to know that the joint intensity of the point process $\theta_n \,\, (1 \le n \le N)$ is given by the determinant of the self-adjoint, positive semidefinite matrix with entries $K(x_i, x_j)$. In particular, for any integer $M \ge 1$ and any Borel measurable function $f: [a,b]^M \to \mathbb{C}$,
\begin{equation}\label{determinantal}
\begin{split} \mathbb{E} \sum_{\substack{n_1, n_2, \ldots, n_M=1 \\ \textrm{pairwise distinct}}}^N f(\theta_{n_1}, &\theta_{n_2}, \ldots, \theta_{n_M}) = \\ &\frac{1}{(b-a)^M} \int_{[a,b]^M} \det (K(x_i, x_j))_{i,j=1}^M f(x_1, x_2, \ldots, x_M) \, \mathrm{d}x_1 \mathrm{d}x_2 \cdots \mathrm{d}x_M \end{split}
\end{equation}
provided that the integral exists.

The exact formula \eqref{wassersteinexact} for the empirical spectral measure $\mu_A$ takes the form
\begin{equation}\label{wassersteinUN}
W_2^2 \left( \mu_A , \lambda_{\mathbb{T}} \right) = \frac{2}{N^2} \sum_{k=1}^{\infty} \frac{1}{k^2} \left| \sum_{n=1}^N e^{i k \theta_n} \right|^2 \qquad \textrm{for } G=\mathrm{U}(N),
\end{equation}
resp.
\begin{equation}\label{wassersteinothergroups}
W_2^2 \left( \mu_A , \lambda_{\mathbb{T}} \right) = \frac{2}{N^2} \sum_{k=1}^{\infty} \frac{1}{k^2} \left( \sum_{n=1}^N \cos (k \theta_n) \right)^2
\end{equation}
for $G=\mathrm{SO}(2N+1)$, $\mathrm{SO}(2N)$, $\mathrm{O}^- (2N+2)$ and $\mathrm{USp}(2N)$.

\section{Exact formulas for the expected value and the variance}\label{exactsection}

In this section, we find exact formulas for the expected value and the variance of $W_2^2 (\mu_A, \lambda_{\mathbb{T}})$. Throughout, $x^+ = \max \{ x, 0 \}$ denotes the positive part of a real number $x$.
\begin{thm}\label{exacttheorem} Let $G= \mathrm{U}(N), \mathrm{SU}(N), \mathrm{SO}(2N+1), \mathrm{O}(2N+1), \mathrm{SO}(2N), \mathrm{O}^- (2N+2)$ or $\mathrm{USp}(2N)$, and let $A \in G$ be a random matrix distributed according to the normalized Haar measure on $G$. We have
\[ \begin{split} \mathbb{E} \, W_2^2 (\mu_A, \lambda_{\mathbb{T}}) &= \frac{2}{N_0^2} \sum_{k=1}^{\infty} \frac{\min \{ k,N_0 \} + \eta (k)}{k^2}, \\ \mathrm{Var} \, W_2^2 (\mu_A, \lambda_{\mathbb{T}}) &= \frac{4}{N_0^4} \sum_{k=1}^{\infty} \frac{T(k)}{k^4} + \frac{4}{N_0^4} \sum_{\substack{k,\ell =1 \\ k \neq \ell}}^{\infty} \frac{V(k,\ell)+\delta (k,\ell)}{k^2 \ell^2} , \end{split} \]
where $\eta (k)$, $T(k)$, $V(k,\ell)$ and $\delta (k,\ell)$ are as follows.
\begin{enumerate}
\item[(i)] Let $G=\mathrm{U}(N)$ or $\mathrm{SU}(N)$. Then
\[ \begin{split} \eta (k) &= 0, \\ T(k) &=\min \{ k^2, N^2 \} + \min \{ 2k, N \} - 2 \min \{ k,N \} , \\ V(k,\ell) &=2(N-\max \{ k,\ell \} )^+ - (N-k-\ell)^+ - (N-|k-\ell|)^+ , \\ \delta (k,\ell) &= 0 . \end{split} \]

\item[(ii)] Let $G=\mathrm{SO}(2N+1)$ or $\mathrm{O}(2N+1)$. For any $a,b,c \in \mathbb{Z}$, define $\varepsilon(a) = \mathds{1}_{\{ 1 \le a \le 2N-1, \,\, a \textrm{ odd} \}}$ and $\alpha (a,b,c) = (\min \{ a,N \} + \min \{ b,N \} -c)^+$. Then
\[ \begin{split} \eta (k) &=\varepsilon (k) , \\ T(k) &= 2 \min \{ k^2, (2N)^2 \} + 4 \varepsilon (k) \min \{ k, 2N \} + 6(\min \{ k, N \} - \min \{ k,2N \}) +4 (\varepsilon (3k) - \varepsilon (k)) , \\ V(k,\ell) &= 8 (N-\max \{ k,\ell \})^+ - 2(2N - |k-\ell|)^+ - 2(2N-k-\ell)^+ \\ &\phantom{={}}+ 4 \alpha (\min \{ k,\ell \}, \max \{ k,\ell \}, \max \{ k,\ell \}) + 4 \alpha (\ell, \ell -k, \ell) + 4 \alpha (k, k-\ell, k) \\ &\phantom{={}}+ 4 \alpha (k+\ell, \ell, k+\ell) + 4\alpha (k+\ell, k, k+\ell) - 8 \alpha (k+\ell, \min \{ k,\ell \}, k+\ell) \\ &\phantom{={}} -4\alpha (k,k,k+\ell) -4 \alpha (\ell, \ell, k+\ell) , \\ \delta(k,\ell) &=2 (\varepsilon (\ell) \varepsilon (|2k-\ell|) - \varepsilon (2k+\ell)) + 2(\varepsilon (k) \varepsilon (|2\ell -k|) - \varepsilon (2\ell +k)) + 2 (\varepsilon (|k-\ell|) - \varepsilon (k+\ell)) \\ &\phantom{={}} +4 \varepsilon (\ell) \varepsilon (|2k -\ell|) + 4 \varepsilon (2k+\ell) + 4 \varepsilon (k) \varepsilon (|2\ell -k|) + 4 \varepsilon (2\ell +k) \\ &\phantom{={}} - 8 \varepsilon (2\max \{ k,\ell \} - \min \{ k,\ell \}) - 8 \varepsilon (\max \{ k,\ell \}) . \end{split} \]

\item[(iii)] Let $G=\mathrm{SO}(2N)$. For any $a,b,c \in \mathbb{Z}$, define $\varepsilon (a) = \mathds{1}_{\{ 1 \le a \le 2N-2, \,\, a \textrm{ even} \}}$ and $\alpha (a,b,c) = (\min \{ a, N \} + \min \{ b,N \} -c-1 )^+$. Then
\[ \begin{split} \eta (k) &=\varepsilon (k) + \mathds{1}_{\{ k \le 2N-1 \}} - \mathds{1}_{\{ k \le N-1 \}}, \\ T(k) &=2 \min \{ (k+1)^2, (2N)^2 \} + (6+ \mathds{1}_{\{ k \le N-1 \}} 2) (2N-k-1)^+ - 6N (1+\mathds{1}_{\{ k \le N-1 \}}) \\ &\phantom{={}} - 4 (N-k)^+ + \mathds{1}_{\{ k \le N-1 \}}6 + 4 \varepsilon (k) \left( \min \{ k+1, 2N \} - \mathds{1}_{\{ k \le N-1 \}} \right) + 4 (\varepsilon (3k) - \varepsilon (k)) , \\ V(k,\ell) &=8 (N - \max \{ k,\ell \} -1)^+ - 2 (2N - |k-\ell|-1)^+ -2 (2N-k-\ell-1)^+ \\ &\phantom{={}}+4 \alpha (\min \{ k,\ell \}, \max \{ k ,\ell \}, \max \{ k ,\ell \}) + 4 \alpha (\ell, \ell-k, \ell) + 4 \alpha (k, k-\ell, k) \\ &\phantom{={}}+4 \alpha (k+\ell, \ell, k+\ell) + 4 \alpha (k+\ell, k, k+\ell) -8 \alpha (k+\ell, \min \{ k ,\ell \}, k+\ell) \\ &\phantom{={}} -4 \alpha (k,k,k+\ell) - 4 \alpha (\ell, \ell, k+\ell) - \mathds{1}_{\{ k+\ell \le N-1 \}} + \mathds{1}_{\{ \max \{ k ,\ell \} \le N-1 \}} 16 \\ &\phantom{={}}+ \mathds{1}_{\{ |k-\ell| \le N-1 \}} (-3+\mathds{1}_{\{ \min \{ k ,\ell \} \le N-1 \}} 16 -\mathds{1}_{\{ k \le N-1 \}} 6 - \mathds{1}_{\{ \ell \le N-1 \}} 6 ) \\ &\phantom{={}} - \mathds{1}_{\{ |2k-\ell|+\ell \le 2N-2 \}} 4 - \mathds{1}_{\{ |2\ell -k|+k \le 2N-2 \}} 4, \\ \delta(k,\ell) &=4 \varepsilon (k) \varepsilon (\ell) \left( 1-\varepsilon (k+\ell) \right) + 2 \left( \varepsilon (|k-\ell|) - \varepsilon (k+\ell) \right) \\ &\phantom{={}}+\mathds{1}_{\{ \ell \neq 2k \}} \left( 6 \varepsilon (\ell) \varepsilon (|2k-\ell|) + 2 \varepsilon (2k+\ell) -8 \varepsilon (|k-\ell| +k) \right) \\ &\phantom{={}} + \mathds{1}_{\{ k \neq 2\ell \}} \left( 6 \varepsilon (k) \varepsilon (|2\ell -k|) +2 \varepsilon (2\ell +k) -8 \varepsilon (|\ell -k| +\ell) \right) . \end{split} \]

\item[(iv)] Let $G=\mathrm{O}^- (2N+2)$ or $\mathrm{USp}(2N)$. For any $a,b,c \in \mathbb{Z}$, define $\varepsilon (a) = \mathds{1}_{\{ 1 \le a \le 2N, \,\, a \textrm{ even} \}}$ and $\alpha (a,b,c) = (\min \{ a-1, N \} + \min (b-1, N) - c+1)^+$. Then
\[ \begin{split} \eta (k) &=\varepsilon (k) + \mathds{1}_{\{ k \le N \}} - \mathds{1}_{\{ k \le 2N \}}, \\ T(k) &=2 \min \{ (k-1)^2, (2N)^2 \} + (6-\mathds{1}_{\{ k \le N \}} 2) (2N-k+1)^+ - 8 (N-k)^+ - \mathds{1}_{\{ k \le N \}} 8 \\ &\phantom{={}} + \mathds{1}_{\{ k \le N/2 \}} 2 + 4 \varepsilon (k) (\min \{ k-1, 2N \} + \mathds{1}_{\{ k \le N \}}) - \mathds{1}_{\{ k > N \}} 6N + 4(\varepsilon (3k) - \varepsilon (k)) , \\ V(k,\ell) &= 8 (N-\max \{ k,\ell \})^+ - 2 (2N-|k-\ell|+1)^+ - 2(2N-k-\ell+1)^+ \\ &\phantom{={}} +4 \alpha (\min \{ k,\ell \}, \max \{ k,\ell \}, \max \{ k,\ell \}) + 4 \alpha (\ell, \ell-k,\ell) + 4 \alpha (k,k-\ell,k) \\ &\phantom{={}} +4 \alpha (k+\ell, \ell, k+\ell) + 4 \alpha (k+\ell, k, k+\ell) - 8 \alpha (k+\ell, \min \{ k,\ell \}, k+\ell) \\ &\phantom{={}} -4 \alpha (k,k,k+\ell) - 4 \alpha (\ell, \ell, k+\ell) + \mathds{1}_{\{ k+\ell \le N \}} + \mathds{1}_{\{ |k-\ell| \le N \}} 3 \\ &\phantom{={}} +\mathds{1}_{\{ |2k-\ell|+\ell \le 2N \}} 4 + \mathds{1}_{\{ |2\ell -k|+k \le 2N \}} 4 - \mathds{1}_{\{ |k-\ell| \le N \}} 2(\mathds{1}_{\{ k \le N \}} + \mathds{1}_{\{ \ell \le N \}}) , \\ \delta (k,\ell) &= 4 \varepsilon (k) \varepsilon (\ell) ( \varepsilon (k+\ell) -1) + 2 (\varepsilon (|k-\ell|) - \varepsilon (k+\ell)) \\ &\phantom{={}}+ \mathds{1}_{\{ \ell \neq 2k \}} \left( 6 \varepsilon (\ell) \varepsilon (|2k-\ell|) + 2 \varepsilon (2k+\ell) - 8 \varepsilon (|k-\ell|+k) \right) \\ &\phantom{={}}+ \mathds{1}_{\{ k \neq 2\ell \}} \left( 6 \varepsilon (k) \varepsilon (|2\ell-k|) + 2 \varepsilon (2\ell+k) - 8 \varepsilon (|\ell-k|+\ell) \right) . \end{split} \]
\end{enumerate}
\end{thm}

\subsection{Proof for $\mathrm{U}(N)$ and $\mathrm{SU}(N)$}

\begin{proof}[Proof of Theorem \ref{exacttheorem} (i)] By Lemma \ref{reductionlemma}, it is enough to prove the theorem for $G=\mathrm{U}(N)$. We rely on the exact formula \eqref{wassersteinUN} for the Wasserstein metric, which we expand to
\begin{equation}\label{wassersteinUNexpand}
W_2^2 (\mu_A, \lambda_{\mathbb{T}}) = \frac{2}{N^2} \sum_{k=1}^{\infty} \frac{1}{k^2} \Bigg( N + \sum_{\substack{n_1, n_2 =1 \\ n_1 \neq n_2}}^N e^{ik (\theta_{n_1} - \theta_{n_2})} \Bigg) .
\end{equation}

The points $\theta_n \,\, (1 \le n \le N)$ form a determinantal point process on $[0,2 \pi]$ with kernel function $K(x,y) = \sum_{j=0}^{N-1} e^{ij(x-y)}$ with respect to the normalized Lebesgue measure on $[0,2 \pi]$. Observe that for any $y,z \in \mathbb{R}$ and any $k \in \mathbb{Z}$,
\begin{equation}\label{KUNformulas}
\begin{split} \frac{1}{2 \pi} \int_0^{2 \pi} K(y,x) e^{ikx} \, \mathrm{d}x &= \mathds{1}_{\{ 0 \le k \le N-1 \}} e^{iky} , \\ \frac{1}{2 \pi} \int_0^{2 \pi} K(x,y) K(z,x) e^{ikx} \, \mathrm{d}x &= \left\{ \begin{array}{ll} \sum_{j=0}^{N-k-1} e^{i(j+k)z - i jy} & \textrm{if } k \ge 0, \\ \sum_{j=0}^{N-|k|-1} e^{ijz - i (j+|k|)y} & \textrm{if } k<0 . \end{array} \right. \end{split}
\end{equation}
The special case $z=y$ reads
\begin{equation}\label{KUNformulasspecial}
\frac{1}{2 \pi} \int_0^{2 \pi} |K(x,y)|^2 e^{ikx} \, \mathrm{d}x = (N-|k|)^+ e^{iky} .
\end{equation}

The main formula \eqref{determinantal} for determinantal point processes and the identity \eqref{KUNformulasspecial} show that for any integer $k \ge 1$,
\[ \mathbb{E} \sum_{\substack{n_1, n_2=1 \\ n_1 \neq n_2}}^N e^{ik(\theta_{n_1} - \theta_{n_2})} = \frac{1}{(2 \pi)^2} \int_0^{2 \pi} \!\! \int_0^{2 \pi} (N^2 - |K(x,y)|^2) e^{ik(x-y)} \, \mathrm{d} x \, \mathrm{d}y = -(N-k)^+ . \]
The exact formula \eqref{wassersteinUNexpand} together with the fact that $N-(N-k)^+=\min \{ k,N \}$ immediately yield
\[ \mathbb{E} \, W_2^2 (\mu_A, \lambda_{\mathbb{T}}) = \frac{2}{N^2} \sum_{k=1}^{\infty} \frac{\min \{ k,N \}}{k^2}, \]
as claimed.

Squaring \eqref{wassersteinUNexpand} and using the previous formula leads to
\begin{equation}\label{secondmomentUN}
\begin{split} \mathbb{E} \, W_2^4 (\mu_A, \lambda_{\mathbb{T}}) &= \frac{4}{N^4} \sum_{k,\ell=1}^{\infty} \frac{1}{k^2 \ell^2} \mathbb{E} \Bigg( N + \sum_{\substack{n_1, n_2=1 \\ n_1 \neq n_2}}^N e^{ik(\theta_{n_1} - \theta_{n_2})} \Bigg) \Bigg( N + \sum_{\substack{m_1, m_2=1 \\ m_1 \neq m_2}}^N e^{ik(\theta_{m_1} - \theta_{m_2})} \Bigg) \\ &= \frac{4}{N^4} \sum_{k,\ell=1}^{\infty} \left( N^2 - N(N-k)^+ - N(N-\ell)^+ + E(k,\ell) \right) , \end{split}
\end{equation}
where
\[ E(k,\ell) = \mathbb{E} \sum_{\substack{n_1, n_2, m_1, m_2=1 \\ n_1 \neq n_2, \,\, m_1 \neq m_2}}^N e^{ik(\theta_{n_1} - \theta_{n_2}) + i \ell (\theta_{m_1} - \theta_{m_2})} . \]
Fix integers $k, \ell \ge 1$, and let us decompose the sum $E(k, \ell)$ based on the cardinality of the index set $\{ n_1, n_2, m_1, m_2 \}$.\\

\noindent\textbf{Case 1.} Consider the terms for which $\{ n_1, n_2, m_1, m_2 \}$ has cardinality 4. The main formula \eqref{determinantal} for determinantal point processes shows that
\begin{multline*}
S_4 \coloneqq \mathbb{E} \sum_{\substack{n_1, n_2, m_1, m_2=1 \\ \textrm{pairwise distinct}}}^N e^{i k(\theta_{n_1} - \theta_{n_2}) + i \ell (\theta_{m_1} - \theta_{m_2})} \\ =\frac{1}{(2 \pi)^4}\int_0^{2 \pi} \!\! \int_0^{2 \pi} \!\! \int_0^{2 \pi} \!\! \int_0^{2 \pi} \det (K(x_i, x_j))_{i,j=1}^4 e^{i k (x_1-x_2) + i \ell (x_3 - x_4)} \, \mathrm{d}x_1 \mathrm{d}x_2 \mathrm{d}x_3 \mathrm{d}x_4 .
\end{multline*}
Let us first integrate with respect to $x_1$. Expanding the $4 \times 4$ determinant along the first column and noting that the subdeterminant corresponding to the entry $K(x_1,x_1)=N$ does not depend on $x_1$ leads to
\[ \frac{1}{2 \pi} \int_0^{2 \pi} \det (K(x_i, x_j))_{i,j=1}^4 e^{i k x_1} \, \mathrm{d}x_1 = I_1 + I_2 + I_3, \]
where, using the identities \eqref{KUNformulas} and \eqref{KUNformulasspecial},
\[ \begin{split} I_1 = I_1 (x_2, x_3, x_4) &= \frac{1}{2 \pi} \int_0^{2 \pi} -K(x_2,x_1) \det \left( \begin{array}{ccc} K(x_1 , x_2) & K(x_1 , x_3) & K(x_1 , x_4) \\ K(x_3 , x_2) & N & K(x_3 , x_4) \\ K(x_4 , x_2) & K(x_4 , x_3) & N  \end{array} \right) e^{i k x_1} \, \mathrm{d}x_1 \\ &= - (N-k)^+ e^{i k x_2} \det \left( \begin{array}{cc} N & K(x_3 , x_4) \\ K(x_4 , x_3) & N \end{array} \right) \\ &\phantom{={}}+\sum_{j=0}^{N-k-1} e^{i(j+k)x_2 - i j x_3} \det \left( \begin{array}{cc} K(x_3 , x_2) & K(x_3 , x_4) \\ K(x_4 , x_2) & N \end{array} \right) \\ &\phantom{={}}-\sum_{j=0}^{N-k-1} e^{i (j+k)x_2 - i j x_4} \det \left( \begin{array}{cc} K(x_3 , x_2) & N \\ K(x_4 , x_2) & K(x_4 , x_3) \end{array} \right) , \\ I_2 = I_2 (x_2, x_3, x_4) &=\frac{1}{2 \pi} \int_0^{2 \pi} K(x_3 , x_1) \det \left( \begin{array}{ccc} K(x_1 , x_2) & K(x_1 , x_3) & K(x_1 , x_4) \\ N & K(x_2 , x_3) & K(x_2 , x_4) \\ K(x_4 , x_2) & K(x_4 , x_3) & N \end{array} \right) e^{i k x_1} \, \mathrm{d}x_1 \\ &= \sum_{j=0}^{N-k-1} e^{i (j+k)x_3 - i j x_2} \det \left( \begin{array}{cc} K(x_2 , x_3) & K(x_2 , x_4) \\ K(x_4 , x_3) & N \end{array} \right) \\ &\phantom{={}}- (N-k)^+ e^{i k x_3} \det \left( \begin{array}{cc} N & K(x_2 , x_4) \\ K(x_4 , x_2) & N \end{array} \right) \\ &\phantom{={}}+\sum_{j=0}^{N-k-1} e^{i(j+k) x_3 - i j x_4} \det \left( \begin{array}{cc} N & K(x_2 , x_3) \\ K(x_4 , x_2) & K(x_4 , x_3) \end{array} \right) , \end{split} \]
and
\[ \begin{split} I_3 = I_3 (x_2, x_3, x_4) &=\frac{1}{2 \pi} \int_0^{2 \pi} -K(x_4 , x_1) \det \left( \begin{array}{ccc} K(x_1 , x_2) & K(x_1 , x_3) & K(x_1 , x_4) \\ N & K(x_2 , x_3) & K(x_2 , x_4) \\ K(x_3 , x_2) & N & K(x_3 , x_4) \end{array} \right) e^{i k x_1} \, \mathrm{d}x_1 \\ &= -\sum_{j=0}^{N-k-1} e^{i(j+k)x_4 - i j x_2} \det \left( \begin{array}{cc} K(x_2 , x_3) & K(x_2 , x_4) \\ N & K(x_3 , x_4) \end{array} \right) \\ &\phantom{={}}+ \sum_{j=0}^{N-k-1} e^{i(j+k)x_4 - i j x_3} \det \left( \begin{array}{cc} N & K(x_2 , x_4) \\ K(x_3 , x_2) & K(x_3 , x_4) \end{array} \right) \\ &\phantom{={}}- (N-k)^+ e^{i k x_4} \det \left( \begin{array}{cc} N & K(x_2 , x_3) \\ K(x_3 , x_2) & N \end{array} \right) . \end{split} \]
Next, we integrate with respect to $x_2$. The identities \eqref{KUNformulas} and \eqref{KUNformulasspecial} show that
\[ \begin{split} J_1 = J_1 (x_3,x_4) \coloneqq \frac{1}{2 \pi} \int_0^{2 \pi} I_1 e^{-i k x_2} \, \mathrm{d}x_2 &= -(N-k)^+ (N^2 - |K(x_3 , x_4)|^2) \\ &\phantom{={}}+ N(N-k)^+ - \sum_{j=0}^{N-k-1} e^{ij(x_4-x_3)} K(x_3 , x_4) \\ &\phantom{={}}+N(N-k)^+ - \sum_{j=0}^{N-k-1} e^{ij(x_3 - x_4)} K(x_4 , x_3) , \\ J_2 = J_2 (x_3, x_4) \coloneqq \frac{1}{2 \pi} \int_0^{2 \pi} I_2 e^{-i k x_2} \, \mathrm{d}x_2 &=N(N-k)^+ - \sum_{j=0}^{N-k-1} e^{i(j+k)(x_3-x_4)} K(x_4 , x_3) \\ &\phantom{={}}+((N-k)^+)^2 e^{ik(x_3-x_4)} \\ &\phantom{={}}-\sum_{j,j'=0}^{N-k-1} e^{i(j-j')(x_3-x_4)}, \end{split} \]
and
\[ \begin{split} J_3 = J_3 (x_3, x_4) \coloneqq \frac{1}{2 \pi} \int_0^{2 \pi} I_3 e^{-i k x_2} \, \mathrm{d}x_2 &=N(N-k)^+ - \sum_{j=0}^{N-k-1} e^{i(j+k)(x_4-x_3)} K(x_3 , x_4) \\ &\phantom{={}}-\sum_{j,j'=0}^{N-k-1} e^{i(j-j')(x_4-x_3)} \\ &\phantom{={}}+((N-k)^+)^2 e^{ik(x_4 - x_3)} . \end{split} \]
As $J_1$, $J_2$ and $J_3$ depend only on $x_3-x_4$, integrating with respect to $x_3$ and $x_4$ gives
\[ \begin{split} \frac{1}{(2 \pi)^2} \int_0^{2 \pi} \!\! &\int_0^{2 \pi} (J_1+J_2+J_3) e^{i \ell (x_3-x_4)} \, \mathrm{d}x_3 \mathrm{d}x_4 = \\ \frac{1}{2 \pi} \int_0^{2 \pi} \Bigg( &(N-k)^+ |K(x,0)|^2 + ((N-k)^+)^2 (e^{ikx} + e^{-ikx}) - \sum_{j=0}^{N-k-1} (e^{-jix} K(x,0) + e^{ijx} K(0,x)) \\ &-\sum_{j=0}^{N-k-1} (e^{i(j+k)x} K(0,x) + e^{-i (j+k)x} K(x,0)) - 2 \sum_{j,j'=0}^{N-k-1} e^{i(j-j')x} \Bigg) e^{i \ell x} \, \mathrm{d}x . \end{split} \]
The identities \eqref{KUNformulas} and \eqref{KUNformulasspecial} finally yield
\begin{equation}\label{UNS4}
S_4 = (N-k)^+ (N-\ell)^+ +((N-k)^+)^2 \mathds{1}_{\{ k=\ell \}} - 4 (N-k-\ell)^+ - 2 (N-\max \{ k,\ell \})^+ .
\end{equation}

\noindent\textbf{Case 2.} Consider the terms for which $\{ n_1, n_2, m_1, m_2 \}$ has cardinality 3. Since $n_1 \neq n_2$ and $m_1 \neq m_2$, there are four possibilities: $n_1=m_1$, $n_1=m_2$, $n_2=m_1$ or $n_2 = m_2$. The main formula \eqref{determinantal} for determinantal point processes shows that
\[ \begin{split} \mathbb{E} \sum_{\substack{n_1, n_2, m_1, m_2 =1 \\ n_1, n_2, m_2 \textrm{ pairwise distinct} \\ n_1=m_1}}^N &e^{ik(\theta_{n_1} - \theta_{n_2}) + i \ell (\theta_{m_1} - \theta_{m_2})} \\ &= \frac{1}{(2 \pi)^3} \int_0^{2 \pi} \!\! \int_0^{2 \pi} \!\! \int_0^{2 \pi} \det (K(x_i , x_j))_{i,j=1}^3 e^{i(k+\ell) x_1 - i k x_2 - i \ell x_3} \, \mathrm{d}x_1 \mathrm{d}x_2 \mathrm{d}x_3 . \end{split} \]
Let us first integrate with respect to $x_1$. Expanding the $3 \times 3$ determinant along the first column and noting that the subdeterminant corresponding to the entry $K(x_1, x_1)=N$ does not depend on $x_1$ leads to
\[ \frac{1}{2 \pi} \int_0^{2 \pi} \det (K(x_i , x_j))_{i,j=1}^3 e^{i(k+\ell) x_1} \, \mathrm{d}x_1 = I_1'+I_2', \]
where, using the identities \eqref{KUNformulas} and \eqref{KUNformulasspecial},
\[ \begin{split} I_1' = I_1'(x_2, x_3) &=\frac{1}{2 \pi} \int_0^{2 \pi} -K(x_2 , x_1) \det \left( \begin{array}{cc} K(x_1 , x_2) & K(x_1 , x_3) \\ K(x_3 , x_2) & N \end{array} \right) e^{i(k+\ell)x_1} \, \mathrm{d}x_1 \\ &= -N(N-k-\ell)^+ e^{i(k+\ell)x_2} + \sum_{j=0}^{N-k-\ell-1} e^{i(j+k+\ell) x_2 - i jx_3} K(x_3 , x_2) , \end{split} \]
and
\[ \begin{split} I_2' = I_2'(x_2, x_3) &= \frac{1}{2 \pi} \int_0^{2 \pi} K(x_3 , x_1) \det \left( \begin{array}{cc} K(x_1 , x_2) & K(x_1 , x_3) \\ N & K(x_2 , x_3) \end{array} \right) e^{i(k+\ell)x_1} \, \mathrm{d}x_1 \\ &=-N(N-k-\ell)^+ e^{i(k+\ell)x_3} + \sum_{j=0}^{N-k-\ell-1} e^{i(j+k+\ell) x_3 - i j x_2} K(x_2 , x_3) . \end{split} \]
Integrating with respect to $x_2$ and $x_3$ gives
\[ \begin{split} \frac{1}{(2 \pi)^2} \int_0^{2 \pi} \!\! \int_0^{2 \pi} (I_1' + I_2') e^{-i k x_2 - i \ell x_3} \, \mathrm{d}x_2 \mathrm{d}x_3 &= \frac{1}{2 \pi} \int_0^{2 \pi} \sum_{j=0}^{N-k-\ell-1} \left( e^{i(j+\ell)x} K(0,x) + e^{-i(j+k)x} K(x,0) \right) \, \mathrm{d}x \\&= 2(N-k-\ell)^+ . \end{split} \]
Similar computations show that the contribution of the other three cases are
\[ \begin{split} n_1=m_2 \quad &\mapsto \quad 2(N-\max \{ k, \ell \})^+ - N(N-k)^+ \mathds{1}_{\{ k=\ell \}} , \\ n_2=m_1 \quad &\mapsto \quad 2(N-\max \{ k, \ell \})^+ - N(N-k)^+ \mathds{1}_{\{ k=\ell \}} , \\ n_2=m_2 \quad &\mapsto \quad 2 (N-k-\ell)^+ . \end{split} \]
Therefore
\begin{equation}\label{UNS3}
\begin{split} S_3 &\coloneqq \mathbb{E} \sum_{\substack{n_1, n_2, m_1, m_2 =1 \\ n_1 \neq n_2, \,\, m_1 \neq m_2 \\ |\{ n_1, n_2, m_1, m_2 \}|=3}}^N e^{ik(\theta_{n_1} - \theta_{n_2}) + i \ell (\theta_{m_1} - \theta_{m_2})} \\ &= 4 (N-k-\ell)^+ + 4 (N-\max \{ k ,\ell \})^+ - 2N (N-k)^+ \mathds{1}_{\{ k = \ell \}} . \end{split}
\end{equation}

\noindent\textbf{Case 3.} Consider the terms for which $\{ n_1, n_2, m_1, m_2 \}$ has cardinality 2. Since $n_1 \neq n_2$ and $m_1 \neq m_2$, there are two possibilities: either $n_1=m_1$, $n_2=m_2$, or $n_1=m_2$, $n_2=m_1$. The main formula \eqref{determinantal} for determinantal point processes and the identity \eqref{KUNformulasspecial} show that
\[ \begin{split} \mathbb{E} \sum_{\substack{n_1, n_2, m_1, m_2 =1 \\ n_1 = m_1 \neq n_2 = m_2}}^N e^{ik(\theta_{n_1} - \theta_{n_2}) + i \ell (\theta_{m_1} - \theta_{m_2})} &= \frac{1}{(2 \pi)^2} \int_0^{2 \pi} \!\! \int_0^{2 \pi} (N^2 - |K(x,y)|^2) e^{i(k+\ell)(x-y)} \, \mathrm{d}x \, \mathrm{d}y \\ &= -(N-k-\ell)^+ , \end{split} \]
and
\[ \begin{split} \mathbb{E} \sum_{\substack{n_1, n_2, m_1, m_2 =1 \\ n_1 = m_2 \neq n_2 = m_1}}^N e^{ik(\theta_{n_1} - \theta_{n_2}) + i \ell (\theta_{m_1} - \theta_{m_2})} &= \frac{1}{(2 \pi)^2} \int_0^{2 \pi} \!\! \int_0^{2 \pi} (N^2 - |K(x,y)|^2) e^{i(k-\ell)(x-y)} \, \mathrm{d}x \, \mathrm{d}y \\ &= N^2 \mathds{1}_{\{ k=\ell \}} -(N-|k-\ell|)^+ . \end{split} \]

Adding \eqref{UNS4}, \eqref{UNS3} and the previous two formulas, and noting that $((N-k)^+)^2 -2N (N-k)^+ +N^2 = \min \{ k^2, N^2 \}$ finally show that
\[ E(k,\ell) = (N-k)^+ (N-\ell)^+ + 2 (N-\max \{ k,\ell \})^+ - (N-k-\ell)^+ - (N-|k-\ell|)^+ + \min \{ k^2, N^2 \} \mathds{1}_{\{ k=\ell \}} . \]
In particular, \eqref{secondmomentUN} yields
\[ \begin{split} \mathbb{E} \, W_2^4 (\mu_A, \lambda_{\mathbb{T}}) &= \frac{4}{N^4} \sum_{k, \ell =1}^{\infty} \frac{(N-(N-k)^+) (N-(N-\ell)^+)}{k^2 \ell^2} \\ &\phantom{={}}+\frac{4}{N^4} \sum_{k=1}^{\infty} \frac{\min \{ k^2, N^2 \} + 2(N-k)^+ - (N-2k)^+ -N}{k^4} \\ &\phantom{={}}+\frac{4}{N^4} \sum_{\substack{k, \ell=1 \\ k \neq \ell}}^{\infty} \frac{2(N-\max \{ k,\ell \} )^+ - (N-k-\ell)^+ - (N-|k-\ell|)^+}{k^2 \ell^2} . \end{split} \]
The first line of the previous formula is
\[ \Bigg( \frac{2}{N^2} \sum_{k=1}^{\infty} \frac{N-(N-k)^+}{k^2} \Bigg)^2 = \left( \mathbb{E} W_2^2 (\mu_A, \lambda_{\mathbb{T}}) \right)^2 , \]
hence
\[ \begin{split} \mathrm{Var} \, W_2^2 (\mu_A, \lambda_{\mathbb{T}}) &= \frac{4}{N^4} \sum_{k=1}^{\infty} \frac{\min \{ k^2, N^2 \} + \min \{ 2k,N \} - 2 \min \{ k,N \}}{k^4} \\ &\phantom{={}}+\frac{4}{N^4} \sum_{\substack{k, \ell=1 \\ k \neq \ell}}^{\infty} \frac{2(N-\max \{ k,\ell \} )^+ - (N-k-\ell)^+ - (N-|k-\ell|)^+}{k^2 \ell^2} , \end{split} \]
as claimed. This finishes the proof of Theorem \ref{exacttheorem} for $G=\mathrm{U}(N)$ and $\mathrm{SU}(N)$.
\end{proof}

\subsection{Proof for $\mathrm{SO}(2N+1)$ and $\mathrm{O}(2N+1)$}

\begin{proof}[Proof of Theorem \ref{exacttheorem} (ii)] By Lemma \ref{reductionlemma}, it is enough to prove the theorem for $G=\mathrm{SO}(2N+1)$. We rely on the exact formula \eqref{wassersteinothergroups} for the Wasserstein metric, which we expand to
\begin{equation}\label{wassersteinothergroupsexpand}
W_2^2 (\mu_A, \lambda_{\mathbb{T}}) = \frac{2}{N^2} \sum_{k=1}^{\infty} \frac{1}{k^2} \Bigg( \frac{N}{2} + \frac{1}{2} \sum_{n=1}^N \cos (2 k \theta_n) + \sum_{\substack{n_1, n_2=1 \\ n_1 \neq n_2}}^N \cos (k \theta_{n_1}) \cos (k \theta_{n_2}) \Bigg) .
\end{equation}
The points $\theta_n \,\, (1 \le n \le N)$ form a determinantal point process on $[0,\pi]$ with kernel function $K(x,y)=2 \sum_{j=0}^{N-1} \sin (\frac{2j+1}{2} x) \sin (\frac{2j+1}{2}y)$ with respect to the normalized Lebesgue measure on $[0,\pi]$.

For any $a,b,c,d \in \mathbb{Z}$, define
\begin{equation}\label{gammaabcd}
\begin{split} \Pi (a) &= \frac{1}{\pi} \int_0^{\pi} K(x,x) \cos (ax) \, \mathrm{d}x, \\ \Pi (a,b) &= \frac{1}{\pi^2} \int_0^{\pi} \!\! \int_0^{\pi} K(x,y)^2 \cos (ax) \cos (by) \, \mathrm{d}x \, \mathrm{d}y, \\ \Pi (a,b,c) &= \frac{1}{\pi^3} \int_0^{\pi} \!\! \int_0^{\pi} \!\! \int_0^{\pi} K(x,y) K(y,z) K(z,x) \cos (ax) \cos (by) \cos (cz) \, \mathrm{d}x \, \mathrm{d}y \, \mathrm{d}z, \\ \Pi (a,b,c,d) &= \begin{aligned}[t] \frac{1}{\pi^4} \int_0^{\pi} \!\! \int_0^{\pi} \!\! \int_0^{\pi} \!\! \int_0^{\pi} &K(x,y) K(y,z) K(z,w) K(w,x) \\ &\times \cos (ax) \cos (by) \cos (cz) \cos (dw) \, \mathrm{d}x \, \mathrm{d}y \, \mathrm{d}z \, \mathrm{d}w. \end{aligned} \end{split}
\end{equation}
Note that these are even in all variables, and invariant under cyclic permutation of the arguments. Define further $\varepsilon(a) = \mathds{1}_{\{ 1 \le a \le 2N-1, \,\, a \textrm{ odd} \}}$ and $\alpha (a,b,c) = (\min \{ a,N \} + \min \{ b,N \} -c)^+$.

\begin{lem}\label{gammalemmaSO2N+1} We have $\Pi (0) = \Pi (0,0)=N$. Let $k, \ell \ge 1$ be integers. We have
\[ \Pi (k) = - \frac{\varepsilon (k)}{2}, \quad \Pi (k,\ell) = \left\{ \begin{array}{ll} \frac{(2N-k)^+}{4} & \textrm{if } k=\ell, \\ - \frac{\varepsilon (k+\ell)}{2} & \textrm{if } k \neq \ell, \end{array} \right. \quad \Pi (0,k) = - \frac{\varepsilon (k)}{2}, \]
and
\[ \begin{split} \Pi (k, \ell , \ell) &= \left\{ \begin{array}{ll} \frac{(2N-k)^+}{8} & \textrm{if } k=2\ell, \\ -\frac{3 \varepsilon (k+2\ell) + \varepsilon (k) \varepsilon (|2\ell -k|)}{8} & \textrm{if } k \neq 2 \ell , \end{array} \right. \\ \Pi (k,\ell, k+\ell) &= \frac{(N-k-\ell)^+}{4} + \frac{\alpha (k+\ell, \ell, k+\ell) + \alpha (k+\ell, k, k+\ell)}{8}, \end{split} \]
as well as $\Pi (0,k,k) = \frac{(2N-k)^+}{4}$ and $\Pi (k,k,k,k) = \frac{(N-k)^+}{4} + \frac{(2N-k)^+}{16}$. If $k \neq \ell$, then
\[ \begin{split} \Pi (k,k,\ell,\ell) &= \frac{(N-\max \{ k,\ell \})^+ + (N-k-\ell)^+}{8} + \frac{\alpha (\min \{ k,\ell \}, \max \{ k,\ell \}, \max \{ k,\ell \})}{16} \\ &\phantom{={}}+ \frac{\alpha (k+\ell, \ell, k+\ell) + \alpha (k+\ell, k, k+\ell) + \alpha (\ell, \ell -k, \ell) + \alpha (k, k-\ell, k)}{16}, \\ \Pi (k,\ell,k,\ell) &= \frac{(N-k-\ell)^+}{4} + \frac{\alpha (k+\ell, \min \{ k,\ell \}, k+\ell)}{4} + \frac{\alpha (k,k,k+\ell) + \alpha (\ell, \ell, k+\ell)}{8} . \end{split} \]
\end{lem}

\begin{proof} Recall the trigonometric identities
\[ \cos x \cos y = \frac{\cos (x-y) + \cos (x+y)}{2} \qquad \textrm{and} \qquad \sin x \sin y = \frac{\cos (x-y) - \cos (x+y)}{2}, \]
and observe that for any $a,b \in \mathbb{Z}$,
\[ \frac{1}{\pi} \int_0^{\pi} \cos (ax) \, \mathrm{d}x = \mathds{1}_{\{ a=0 \}} \qquad \textrm{and} \qquad \frac{1}{\pi} \int_0^{\pi} \cos (ax) \cos (bx) \, \mathrm{d}x = \frac{1}{2} \mathds{1}_{\{ a=b \}} + \frac{1}{2} \mathds{1}_{\{ a = -b \}} . \]
The identities $\Pi (0) = \Pi (0,0)=N$ are straightforward. Fix integers $k, \ell \ge 1$.

We have
\[ \Pi (k) = \sum_{j=0}^{N-1} \frac{1}{\pi} \int_0^{\pi} (1-\cos ((2j+1)x)) \cos (k x) \, \mathrm{d}x = \sum_{j=0}^{N-1} \frac{-1}{2} \mathds{1}_{\{ 2j+1 =k \}} = - \frac{\varepsilon (k)}{2}, \]
as claimed. Consider now
\[ \begin{split} \Pi (k,\ell) &= \begin{aligned}[t] \sum_{j,j'=0}^{N-1} \frac{1}{\pi^2} \int_0^{\pi} \!\! \int_0^{\pi} &\left( \cos ((j-j') x ) - \cos ((j+j'+1)x) \right) \\ &\times \left( \cos ((j-j') y ) - \cos ((j+j'+1)y) \right) \cos (k x) \cos (\ell y) \, \mathrm{d}x \, \mathrm{d}y \end{aligned} \\ &= \sum_{j,j'=0}^{N-1} \frac{1}{4} \left( \mathds{1}_{\{ |j-j'| =k \}} - \mathds{1}_{\{ j+j'+1 =k \}} \right) \left( \mathds{1}_{\{ |j-j'| =\ell \}} - \mathds{1}_{\{ j+j'+1 =\ell \}} \right) . \end{split} \]
If $k=\ell$, then
\[ \Pi (k,\ell) = \sum_{j,j'=0}^{N-1} \frac{1}{4} \left( \mathds{1}_{\{ |j-j'| =k \}} + \mathds{1}_{\{ j+j'+1 =k \}} \right) = \frac{(2N-k)^+}{4}, \]
as claimed. If $k \neq \ell$, then
\[ \begin{split} \left. \begin{split} |j-j'|&=k \\ j+j'+1&=\ell \end{split} \right\} &\Longleftrightarrow \,\, \{ 2j+1, 2j'+1 \} = \{ \ell +k, \ell -k \} , \\ \left. \begin{split} |j-j'|&=\ell \\ j+j'+1&=k \end{split} \right\} &\Longleftrightarrow \,\, \{ 2j+1, 2j'+1 \} = \{ k+\ell, k-\ell \} . \end{split} \]
Hence
\[ \Pi (k,\ell) = \sum_{j,j'=0}^{N-1} \frac{-1}{4} \mathds{1}_{\{ \{ 2j+1, 2j'+1 \} = \{ \ell +k, |\ell -k| \} \}} = - \frac{\varepsilon (k+\ell)}{2}, \]
as claimed. Further,
\[ \Pi (0,k) = \sum_{j,j'=0}^{N-1} \frac{1}{2} \mathds{1}_{\{ j=j' \}} \left( \mathds{1}_{\{ |j-j'|=k \}} - \mathds{1}_{\{ j+j'+1=k \}} \right) = \sum_{j=0}^{N-1} \frac{-1}{2} \mathds{1}_{\{ 2j+1=k \}} = - \frac{\varepsilon (k)}{2}, \]
as claimed.

We similarly deduce
\begin{multline*} \Pi (k, \ell, \ell) = \\ \sum_{j,j',j''=0}^{N-1} \frac{1}{8} \left( \mathds{1}_{\{ |j-j'|=k \}} - \mathds{1}_{\{ j+j'+1=k \}} \right) \left( \mathds{1}_{\{ |j'-j''|=\ell \}} - \mathds{1}_{\{ j'+j''+1=\ell \}} \right) \left( \mathds{1}_{\{ |j''-j|=\ell \}} - \mathds{1}_{\{ j''+j+1=\ell \}} \right) .
\end{multline*}
Expanding the triple product gives $8$ terms. Assume first that $k=2\ell$. The cases
\[ \left. \begin{split} |j-j'|&=k \\ |j'-j''| &= \ell \\ j'' + j +1&=\ell \end{split} \right\} \qquad \left. \begin{split} |j-j'|&=k \\ j'+j''+1 &= \ell \\ |j'' - j| &=\ell \end{split} \right\} \qquad \left. \begin{split} j+j' +1&=k \\ |j'-j''| &= \ell \\ |j'' - j|&=\ell \end{split} \right\} \qquad \left. \begin{split} j+j'+1&=k \\ j'+j''+1 &= \ell \\ j'' + j +1&=\ell \end{split} \right\} \]
are impossible for parity reasons, whereas $|j-j'|=k$, $j'+j''+1= \ell$, $j'' + j +1=\ell$ is impossible as the latter two equations imply $j=j'$, but by assumption $k \neq 0$. The remaining three cases are
\[ \begin{split}
\left. \begin{split} |j-j'|&=k \\ |j'-j''|&=\ell \\ |j''-j| &= \ell \end{split} \right\} \, &\Longleftrightarrow \, j, j'', j' \textrm{ is an arithmetic progression with difference } \pm \ell , \\ \left. \begin{split} j+j'+1&=k \\ |j'-j''|&=\ell \\ j''+j+1 &= \ell \end{split} \right\} \, &\Longleftrightarrow \, \left( \begin{array}{c} j \\ j' \\ j'' \end{array} \right) = \left( \begin{array}{c} j \\ 2\ell -1-j \\ \ell -1-j \end{array} \right), \\ \left. \begin{split} j+j'+1&=k \\ j'+j''+1&=\ell \\ |j''-j| &= \ell \end{split} \right\} \, &\Longleftrightarrow \, \left( \begin{array}{c} j \\ j' \\ j'' \end{array} \right) =  \left( \begin{array}{c} 2\ell -1-j' \\ j' \\ \ell-1-j' \end{array} \right).
\end{split} \]
Elementary calculations lead to the desired formula for $\Pi (k, \ell, \ell)$, $k =2\ell$.

Assume next that $k \neq 2\ell$. Then the cases
\[ \left. \begin{split} |j-j'|&=k \\ |j'-j''| &= \ell \\ |j'' - j|&=\ell \end{split} \right\} \qquad \left. \begin{split} |j-j'|&=k \\ j'+j''+1 &= \ell \\ j'' + j+1 &=\ell \end{split} \right\} \qquad \left. \begin{split} j+j' +1&=k \\ |j'-j''| &= \ell \\ j'' + j+1&=\ell \end{split} \right\} \qquad \left. \begin{split} j+j'+1&=k \\ j'+j''+1 &= \ell \\ |j'' - j|&=\ell \end{split} \right\} \]
are impossible. The remaining four cases, after discarding negative solutions are
\[ \begin{split}
\left. \begin{split} |j-j'|&=k \\ |j'-j''|&=\ell \\ j''+j+1 &= \ell \end{split} \right\} \, &\Longleftrightarrow \, \left( \begin{array}{c} 2j+1 \\ 2j'+1 \\ 2j''+1 \end{array} \right) = \left( \begin{array}{c} 2 \ell -k \\ 2 \ell +k \\ k \end{array} \right)  , \\ \left. \begin{split} |j-j'|&=k \\ j'+j''+1&=\ell \\ |j''-j| &= \ell \end{split} \right\} \, &\Longleftrightarrow \, \left( \begin{array}{c} 2j+1 \\ 2j'+1 \\ 2j''+1 \end{array} \right) = \left( \begin{array}{c} 2 \ell +k \\ 2 \ell -k \\ k \end{array} \right)  , \\ \left. \begin{split} j+j'+1&=k \\ |j'-j''|&=\ell \\ |j''-j| &= \ell \end{split} \right\} \, &\Longleftrightarrow \, \left( \begin{array}{c} 2j+1 \\ 2j'+1 \\ 2j''+1 \end{array} \right) \in \left\{ \left( \begin{array}{c} k-2 \ell \\ k+2 \ell \\ k \end{array} \right) , \left( \begin{array}{c} k+2 \ell \\ k-2 \ell \\ k \end{array} \right) , \left( \begin{array}{c} k \\ k \\ k+2\ell \end{array} \right) , \left( \begin{array}{c} k \\ k \\ k - 2 \ell \end{array} \right) \right\} , \\ \left. \begin{split} j+j'+1&=k \\ j'+j''+1&=\ell \\ j''+j+1 &= \ell \end{split} \right\} \, &\Longleftrightarrow \, \left( \begin{array}{c} 2j+1 \\ 2j'+1 \\ 2j''+1 \end{array} \right) = \left( \begin{array}{c} k \\ k \\ 2\ell -k \end{array} \right) .
\end{split} \]
Elementary calculations lead to the desired formula for $\Pi (k,\ell, \ell)$, $k \neq 2\ell$.

We similarly deduce
\begin{multline*} \Pi (k, \ell, k+\ell) = \\ \sum_{j,j',j''=0}^{N-1} \frac{1}{8} \left( \mathds{1}_{\{ |j-j'|=k \}} - \mathds{1}_{\{ j+j'+1=k \}} \right) \left( \mathds{1}_{\{ |j'-j''|=\ell \}} - \mathds{1}_{\{ j'+j''+1=\ell \}} \right) \left( \mathds{1}_{\{ |j''-j|=k+\ell \}} - \mathds{1}_{\{ j''+j+1=k+\ell \}} \right) .
\end{multline*}
Expanding the triple product gives $8$ terms. The cases
\[ \left. \begin{split} |j-j'|&=k \\ |j'-j''| &= \ell \\ j'' + j +1&=k+\ell \end{split} \right\} \qquad \left. \begin{split} |j-j'|&=k \\ j'+j''+1 &= \ell \\ |j'' - j| &=k+\ell \end{split} \right\} \qquad \left. \begin{split} j+j' +1&=k \\ |j'-j''| &= \ell \\ |j'' - j|&=k+\ell \end{split} \right\} \qquad \left. \begin{split} j+j'+1&=k \\ j'+j''+1 &= \ell \\ j'' + j +1&=k+\ell \end{split} \right\} \]
are impossible for parity reasons. The case $j+j'+1=k$, $j'+j''+1=\ell$, $|j''-j|=k+\ell$ is also impossible, as $|\ell -k| \neq k+\ell$. The remaining three cases are
\[ \begin{split}
\left. \begin{split} |j-j'|&=k \\ |j'-j''|&=\ell \\ |j''-j| &= k+\ell \end{split} \right\} \, &\Longleftrightarrow \, \left( \begin{array}{c} j \\ j' \\ j'' \end{array} \right) \in \left\{ \left( \begin{array}{c} j \\ j+k \\ j+k+\ell \end{array} \right) , \left( \begin{array}{c} j \\ j-k \\ j-k-\ell \end{array} \right) \right\} , \\ \left. \begin{split} |j-j'|&=k \\ j'+j''+1&=\ell \\ j''+j+1 &=k+\ell \end{split} \right\} \, &\Longleftrightarrow \, \left( \begin{array}{c} j \\ j' \\ j'' \end{array} \right) = \left( \begin{array}{c} j \\ j-k \\ k+\ell-1-j \end{array} \right) , \\ \left. \begin{split} j+j'+1&=k \\ |j'-j''|&=\ell \\ j''+j+1 &=k+ \ell \end{split} \right\} \, &\Longleftrightarrow \, \left( \begin{array}{c} j \\ j' \\ j'' \end{array} \right) = \left( \begin{array}{c} j \\ k-1-j \\ k+\ell-1-j \end{array} \right) .
\end{split} \]
As an example, note that
\[ \begin{split} \sum_{j,j',j''=0}^{N-1} \mathds{1}_{\left\{ \left( \begin{array}{c} j \\ j' \\ j'' \end{array} \right) = \left( \begin{array}{c} j \\ j-k \\ k+\ell-1-j \end{array} \right) \right\}} &= | [0,N-1] \cap [k,N+k-1] \cap [k+\ell -N, k+\ell -1] \cap \mathbb{Z}| \\ &= |[\max \{ k,k+\ell-N \}, \min \{ N-1, k+\ell-1 \}] \cap \mathbb{Z}| \\ &= (\min \{ N, k+\ell \} - \max \{ k,k+\ell -N \} )^+ \\ &= (\min \{ N, k+\ell \}  + \min \{ -k,N-k-\ell \})^+ \\ &= ( \min \{ N,k+\ell \} + \min \{ \ell, N \} - k-\ell)^+ \\ &= \alpha (k+\ell, \ell, k+\ell) . \end{split} \]
Similar computations lead to the desired formula for $\Pi (k,\ell,k+\ell)$. Further,
\[ \Pi (0,k,k) = \sum_{j,j',j''=0}^{N-1} \frac{1}{4} \mathds{1}_{\{ j=j' \}} \left( \mathds{1}_{\{ |j'-j''|=k \}} - \mathds{1}_{\{ j'+j''+1=k \}} \right) \left( \mathds{1}_{\{ |j''-j|=k \}} - \mathds{1}_{\{ j''+j+1=k \}} \right) = \Pi (k,k), \]
as claimed.

We similarly deduce
\[ \begin{split} \Pi (k,k,\ell,\ell) = \sum_{j_1, j_2, j_3, j_4=0}^{N-1} &\frac{1}{16} \left( \mathds{1}_{\{ |j_1-j_2|=k \}} - \mathds{1}_{\{ j_1+j_2+1=k \}} \right) \left( \mathds{1}_{\{ |j_2-j_3|=k \}} - \mathds{1}_{\{ j_2+j_3+1=k \}} \right) \\ &\times \left( \mathds{1}_{\{ |j_3-j_4|=\ell \}} - \mathds{1}_{\{ j_3+j_4+1=\ell \}} \right) \left( \mathds{1}_{\{ |j_4-j_1|=\ell \}} - \mathds{1}_{\{ j_4+j_1+1=\ell \}} \right) . \end{split} \]
Expanding the quadruple product gives 16 terms. Every case containing an odd number of equations of the form $j+j'+1=a$ is impossible for parity reasons. The cases
\[ \left. \begin{split} j_1+j_2+1&=k \\ |j_2-j_3|&=k \\ j_3 + j_4+1&=\ell \\ |j_4 - j_1|&=\ell \end{split} \right\} \qquad \left. \begin{split} |j_1 - j_2|&=k \\ j_2 + j_3+1&=k \\ |j_3-j_4|&=\ell \\ j_4+j_1+1&=\ell \end{split} \right\} \]
are also impossible, as they imply $j_1, j_2, j_3, j_4 < \max \{ k,\ell \}$, but then an equation with absolute value cannot hold.

Assume first that $k=\ell$. The remaining six cases are
\[ \begin{split} \left. \begin{split} |j_1 - j_2|&=k \\ |j_2-j_3|&=k \\ |j_3-j_4|&=k \\ |j_4 - j_1|&=k \end{split} \right\} &\Longleftrightarrow \, \left( \begin{array}{c} j_1 \\ j_2 \\ j_3 \\ j_4 \end{array} \right) \in \left\{ \begin{split} &\left( \begin{array}{c} j_1 \\ j_1+k \\ j_1+2k \\ j_1+k \end{array} \right) , \left( \begin{array}{c} j_1 \\ j_1+k \\ j_1 \\ j_1+k \end{array} \right), \left( \begin{array}{c} j_1 \\ j_1+k \\ j_1 \\ j_1-k \end{array} \right) , \\ &\left( \begin{array}{c} j_1 \\ j_1-k \\ j_1-2k \\ j_1-k \end{array} \right) ,  \left( \begin{array}{c} j_1 \\ j_1-k \\ j_1 \\ j_1-k \end{array} \right) ,  \left( \begin{array}{c} j_1 \\ j_1-k \\ j_1 \\ j_1+k \end{array} \right) \end{split} \right\} , \\ \left. \begin{split} |j_1 - j_2|&=k \\ |j_2-j_3|&=k \\ j_3 + j_4 +1&=k \\ j_4 + j_1 +1&=k \end{split} \right\} &\Longleftrightarrow \, \left( \begin{array}{c} j_1 \\ j_2 \\ j_3 \\ j_4 \end{array} \right) = \left( \begin{array}{c} j_1 \\ j_1+k \\ j_1 \\ k-1-j_1 \end{array} \right), \\ \left. \begin{split} j_1 + j_2 +1&=k \\ j_2 + j_3 +1&=k \\ |j_3-j_4|&=k \\ |j_4 - j_1|&=k \end{split} \right\} &\Longleftrightarrow \, \left( \begin{array}{c} j_1 \\ j_2 \\ j_3 \\ j_4 \end{array} \right) = \left( \begin{array}{c} j_1 \\ k-1-j_1 \\ j_1 \\ j_1+k \end{array} \right) ,  \end{split} \]

\[ \begin{split} \left. \begin{split} j_1 + j_2 +1&=k \\ j_2 + j_3 +1&=k \\ j_3 + j_4 +1&=k \\ j_4 + j_1 +1&=k \end{split} \right\} &\Longleftrightarrow \, \left( \begin{array}{c} j_1 \\ j_2 \\ j_3 \\ j_4 \end{array} \right) = \left( \begin{array}{c} j_1 \\ k-1-j_1 \\ j_1 \\ k-1-j_1 \end{array} \right) , \\ \left. \begin{split} j_1 + j_2 +1&=k \\ |j_2 - j_3| &=k \\ |j_3 - j_4| &=k \\ j_4 + j_1 +1&=k \end{split} \right\} &\Longleftrightarrow \, \left( \begin{array}{c} j_1 \\ j_2 \\ j_3 \\ j_4 \end{array} \right) = \left( \begin{array}{c} j_1 \\ k-1-j_1 \\ 2k-1-j_1 \\ k-1-j_1 \end{array} \right) , \\ \left. \begin{split} |j_1 - j_2| &=k \\ j_2 + j_3 +1&=k \\ j_3 + j_4 +1&=k \\ |j_4 - j_1| &=k \end{split} \right\} &\Longleftrightarrow \, \left( \begin{array}{c} j_1 \\ j_2 \\ j_3 \\ j_4 \end{array} \right) =\left( \begin{array}{c} j_1 \\ j_1-k \\ 2k-1-j_1 \\ j_1-k \end{array} \right) . \end{split} \]
Elementary calculations lead to the desired formula for $\Pi (k,k,k,k)$.

Assume next that $k \neq \ell$. The cases
\[ \left. \begin{split} j_1+j_2+1&=k \\ |j_2-j_3|&=k \\ |j_3-j_4|&=\ell \\ j_4+j_1+1&=\ell \end{split} \right\} \qquad \left. \begin{split} |j_1 - j_2|&=k \\ j_2 + j_3 +1&=k \\ j_3 + j_4 +1&=\ell \\ | j_4 - j_1|&=\ell \end{split} \right\} \]
are impossible by the assumption $k \neq \ell$. The remaining four cases are
\[ \begin{split} \left. \begin{split} |j_1 - j_2|&=k \\ |j_2-j_3|&=k \\ |j_3-j_4|&=\ell \\ |j_4 - j_1|&=\ell \end{split} \right\} &\Longleftrightarrow \, \left( \begin{array}{c} j_1 \\ j_2 \\ j_3 \\ j_4 \end{array} \right) \in \left\{ \left( \begin{array}{c} j_1 \\ j_1 + k \\ j_1 \\ j_1+\ell \end{array} \right), \left( \begin{array}{c} j_1 \\ j_1 - k \\ j_1 \\ j_1+\ell \end{array} \right), \left( \begin{array}{c} j_1 \\ j_1 + k \\ j_1 \\ j_1-\ell \end{array} \right), \left( \begin{array}{c} j_1 \\ j_1 - k \\ j_1 \\ j_1-\ell \end{array} \right) \right\} , \\ \left. \begin{split} |j_1 - j_2|&=k \\ |j_2-j_3|&=k \\ j_3 + j_4 +1&=\ell \\ j_4 + j_1 +1&=\ell \end{split} \right\} &\Longleftrightarrow \, \left( \begin{array}{c} j_1 \\ j_2 \\ j_3 \\ j_4 \end{array} \right) \in \left\{ \left( \begin{array}{c} j_1 \\ j_1+k \\ j_1 \\ \ell-1-j_1 \end{array} \right), \left( \begin{array}{c} j_1 \\ j_1-k \\ j_1 \\ \ell-1-j_1 \end{array} \right) \right\} , \\ \left. \begin{split} j_1 + j_2 +1&=k \\ j_2 + j_3 +1&=k \\ |j_3-j_4|&=\ell \\ |j_4 - j_1|&=\ell \end{split} \right\} &\Longleftrightarrow \, \left( \begin{array}{c} j_1 \\ j_2 \\ j_3 \\ j_4 \end{array} \right) \in \left\{ \left( \begin{array}{c} j_1 \\ k-1-j_1 \\ j_1 \\ j_1+\ell \end{array} \right) , \left( \begin{array}{c} j_1 \\ k-1-j_1 \\ j_1 \\ j_1 - \ell \end{array} \right) \right\} , \\ \left. \begin{split} j_1 + j_2 +1&=k \\ j_2 + j_3 +1&=k \\ j_3 + j_4 +1&=\ell \\ j_4 + j_1 +1&=\ell \end{split} \right\} &\Longleftrightarrow \, \left( \begin{array}{c} j_1 \\ j_2 \\ j_3 \\ j_4 \end{array} \right) =\left( \begin{array}{c} j_1 \\ k-1-j_1 \\ j_1 \\ \ell -1- j_1 \end{array} \right) . \end{split} \]
Elementary calculations lead to the desired formula for $\Pi (k,k,\ell,\ell)$, $k \neq \ell$.

We finally deduce
\[ \begin{split} \Pi (k,\ell,k,\ell) = \sum_{j_1, j_2, j_3, j_4=0}^{N-1} &\frac{1}{16} \left( \mathds{1}_{\{ |j_1-j_2|=k \}} - \mathds{1}_{\{ j_1+j_2+1=k \}} \right) \left( \mathds{1}_{\{ |j_2-j_3|=\ell \}} - \mathds{1}_{\{ j_2+j_3+1=\ell \}} \right) \\ &\times \left( \mathds{1}_{\{ |j_3-j_4|=k \}} - \mathds{1}_{\{ j_3+j_4+1=k \}} \right) \left( \mathds{1}_{\{ |j_4-j_1|=\ell \}} - \mathds{1}_{\{ j_4+j_1+1=\ell \}} \right) . \end{split} \]
Assume that $k \neq \ell$. Expanding the quadruple product gives 16 terms. Every case containing an odd number of equations of the form $j+j'+1=a$ is impossible for parity reasons. The case
\[ \left. \begin{split} j_1 + j_2 +1&=k \\ j_2 + j_3 +1&=\ell \\ j_3 + j_4 +1&=k \\ j_4 + j_1 +1&=\ell \end{split} \right\} \]
is impossible by the assumption $k \neq \ell$. The remaining seven cases are
\[ \begin{split} \left. \begin{split} |j_1 - j_2|&=k \\ |j_2-j_3|&=\ell \\ |j_3-j_4|&=k \\ |j_4 - j_1|&=\ell \end{split} \right\} &\Longleftrightarrow \, \left( \begin{array}{c} j_1 \\ j_2 \\ j_3 \\ j_4 \end{array} \right) \in \left\{ \left( \begin{array}{c} j_1 \\ j_1+k \\ j_1+k+\ell \\ j_1+\ell \end{array} \right) , \left( \begin{array}{c} j_1 \\ j_1+k \\ j_1+k-\ell \\ j_1-\ell \end{array} \right) , \left( \begin{array}{c} j_1 \\ j_1-k \\ j_1-k+\ell \\ j_1+\ell \end{array} \right) , \left( \begin{array}{c} j_1 \\ j_1-k \\ j_1-k-\ell \\ j_1-\ell \end{array} \right) \right\} , \\ \left. \begin{split} |j_1 - j_2|&=k \\ |j_2-j_3|&=\ell \\ j_3 + j_4 +1&=k \\ j_4 + j_1 +1&=\ell \end{split} \right\} &\Longleftrightarrow \, \left( \begin{array}{c} j_1 \\ j_2 \\ j_3 \\ j_4 \end{array} \right) = \left( \begin{array}{c} j_1 \\ j_1+k \\ j_1+k-\ell \\ \ell-1-j_1 \end{array} \right) , \\ \left. \begin{split} j_1 + j_2 +1&=k \\ j_2 + j_3 +1&=\ell \\ |j_3-j_4|&=k \\ |j_4 - j_1|&=\ell \end{split} \right\} &\Longleftrightarrow \, \left( \begin{array}{c} j_1 \\ j_2 \\ j_3 \\ j_4 \end{array} \right) = \left( \begin{array}{c} j_1 \\ k-1-j_1 \\ j_1+\ell -k \\ j_1 + \ell \end{array} \right) , \\ \left. \begin{split} j_1 + j_2 +1&=k \\ |j_2 - j_3|&=\ell \\ |j_3 - j_4| &=k \\ j_4 + j_1 +1&=\ell \end{split} \right\} &\Longleftrightarrow \, \left( \begin{array}{c} j_1 \\ j_2 \\ j_3 \\ j_4 \end{array} \right) =\left( \begin{array}{c} j_1 \\ k-1-j_1 \\ k+\ell -1-j_1 \\ \ell -1-j_1 \end{array} \right) , \\ \left. \begin{split} |j_1 - j_2| &=k \\ j_2 + j_3+1&=\ell \\ j_3 + j_4+1 &=k \\ |j_4 - j_1|&=\ell \end{split} \right\} &\Longleftrightarrow \, \left( \begin{array}{c} j_1 \\ j_2 \\ j_3 \\ j_4 \end{array} \right) = \left( \begin{array}{c} j_1 \\ j_1-k \\ k+\ell -1-j_1 \\ j_1 - \ell \end{array} \right) , \\ \left. \begin{split} j_1 + j_2+1 &=k \\ |j_2 - j_3|&=\ell \\ j_3 + j_4+1 &=k \\ |j_4 - j_1|&=\ell \end{split} \right\} &\Longleftrightarrow \, \left( \begin{array}{c} j_1 \\ j_2 \\ j_3 \\ j_4 \end{array} \right) \in \left\{ \left( \begin{array}{c} j_1 \\ k-1-j_1 \\ k+\ell -1-j_1 \\ j_1 - \ell \end{array} \right) , \left( \begin{array}{c} j_1 \\ k-1-j_1 \\ k-\ell -1-j_1 \\ j_1 + \ell \end{array} \right) \right\} , \\ \left. \begin{split} |j_1 - j_2| &=k \\ j_2 + j_3+1&=\ell \\| j_3 - j_4| &=k \\ j_4 + j_1 +1&=\ell \end{split} \right\} &\Longleftrightarrow \, \left( \begin{array}{c} j_1 \\ j_2 \\ j_3 \\ j_4 \end{array} \right) \in \left\{ \left( \begin{array}{c} j_1 \\ j_1 -k \\ k+\ell -1-j_1 \\ \ell - 1 - j_1 \end{array} \right) , \left( \begin{array}{c} j_1 \\ j_1 +k \\ \ell -k -1-j_1 \\ \ell -1-j_1 \end{array} \right) \right\} . \end{split} \]
Elementary calculations lead to the desired formula for $\Pi (k,\ell,k,\ell)$, $k \neq \ell$. This finishes the proof of Lemma \ref{gammalemmaSO2N+1}.
\end{proof}

The main formula \eqref{determinantal} for determinantal point processes and Lemma \ref{gammalemmaSO2N+1} yield that for any integer $k \ge 1$,
\[ \mathbb{E} \sum_{n=1}^N \cos (2 k \theta_n) = \frac{1}{\pi} \int_0^{\pi} K(x,x) \cos (2kx) \, \mathrm{d}x = \Pi (2k) = 0, \]
and
\[ \begin{split} \mathbb{E} \sum_{\substack{n_1, n_2=1 \\ n_1 \neq n_2}}^N \cos (k \theta_{n_1}) \cos (k \theta_{n_2}) &=\frac{1}{\pi^2} \int_0^{\pi} \!\! \int_0^{\pi} (K(x,x) K(y,y) - K(x,y)^2) \cos (kx) \cos (ky) \, \mathrm{d}x \, \mathrm{d}y \\ &= \Pi (k)^2 - \Pi (k,k) = \frac{\varepsilon (k) - (2N-k)^+}{4} . \end{split} \]
Using \eqref{wassersteinothergroupsexpand}, the previous two formulas and the fact $2N - (2N-k)^+ = \min \{ k,2N \}$ immediately show that
\[ \mathbb{E} \, W_2^2 (\mu_A, \lambda_{\mathbb{T}}) = \frac{2}{(2N)^2} \sum_{k=1}^{\infty} \frac{\min \{ k, 2N \} + \varepsilon (k)}{k^2}, \]
as claimed.

Squaring \eqref{wassersteinothergroupsexpand} and taking the expected value gives
\begin{equation}\label{secondmomentSO2N+1}
\begin{split} \mathbb{E} \, W_2^4 (\mu_A, \lambda_{\mathbb{T}}) = \frac{4}{N^4} \sum_{k, \ell =1}^{\infty} \frac{1}{k^2 \ell^2} \Bigg( &\frac{N^2}{4} + \frac{N}{2} \cdot \frac{\varepsilon (k) -(2N-k)^+}{4} +\frac{N}{2} \cdot \frac{\varepsilon (\ell) -(2N-\ell)^+}{4} \\ &+E(k,\ell) + F(k,\ell) + F(\ell, k)+ H(k,\ell) \Bigg) , \end{split}
\end{equation}
where
\[ \begin{split} E(k,\ell) &= \mathbb{E} \, \frac{1}{4} \sum_{n,m=1}^N \cos (2 k \theta_n) \cos (2 \ell \theta_m) , \\ F(k,\ell) &= \mathbb{E} \, \frac{1}{2} \sum_{\substack{n,m_1, m_2=1 \\ m_1 \neq m_2}}^N \cos (2 k \theta_n) \cos (\ell \theta_{m_1}) \cos (\ell \theta_{m_2}) , \\ H(k,\ell) &= \mathbb{E} \sum_{\substack{n_1, n_2, m_1, m_2 =1 \\ n_1 \neq n_2, \,\, m_1 \neq m_2}}^N \cos (k \theta_{n_1}) \cos (k \theta_{n_2}) \cos (\ell \theta_{m_1}) \cos (\ell \theta_{m_2}) .  \end{split} \]
Fix integers $k, \ell \ge 1$. We will use the main formula \eqref{determinantal} for determinantal point processes and Lemma \ref{gammalemmaSO2N+1} to compute $E(k,\ell)$, $F(k,\ell)$ and $H(k,\ell)$.

To compute $E(k,\ell)$, we consider the contribution of the terms $n=m$,
\[ \begin{split} \mathbb{E} \, \frac{1}{4} \sum_{n=1}^N \cos (2 k \theta_n) \cos (2 \ell \theta_n) &= \frac{1}{4 \pi} \int_0^{\pi} K(x,x) \frac{\cos ((2k-2\ell) x) + \cos ((2k+2\ell)x)}{2} \, \mathrm{d}x \\ &= \frac{\Pi (2k-2\ell)}{8} + \frac{\Pi (2k+2\ell)}{8} = \mathds{1}_{\{ k = \ell \}} \frac{N}{8} , \end{split} \]
and the terms $n \neq m$,
\[ \begin{split} \mathbb{E} \, \frac{1}{4} \sum_{\substack{n,m=1 \\ n \neq m}}^N \cos (2 k \theta_n) \cos (2 \ell \theta_m) &= \frac{1}{4 \pi^2} \int_0^{\pi} \!\! \int_0^{\pi} \left( K(x,x) K(y,y) - K(x,y)^2 \right) \cos (2kx) \cos (2\ell y) \, \mathrm{d}x \, \mathrm{d}y \\ &= \frac{\Pi (2k) \Pi (2\ell)}{4} - \frac{\Pi (2k, 2\ell)}{4}  = - \mathds{1}_{\{ k=\ell \}} \frac{(N-k)^+}{8} . \end{split} \]
Thus
\begin{equation}\label{Ekell}
E(k,\ell) = \mathds{1}_{\{ k = \ell \}} \frac{\min \{ k, N \}}{8} .
\end{equation}

Consider now $F(k,\ell)$. The terms with $n=m_1$ contribute
\[ \begin{split} \mathbb{E} \, \frac{1}{2} &\sum_{\substack{n,m_1, m_2=1 \\ n=m_1 \neq m_2}}^N \cos (2 k \theta_n) \cos (\ell \theta_{m_1}) \cos (\ell \theta_{m_2}) \\ &= \frac{1}{2 \pi^2} \int_0^{\pi} \!\! \int_0^{\pi} (K(x,x) K(y,y) - K(x,y)^2) \cos (2kx) \cos (\ell x) \cos (\ell y) \, \mathrm{d}x \, \mathrm{d}y \\ &= \frac{\Pi (\ell) \Pi (2k-\ell) + \Pi (\ell) \Pi (2k+\ell)}{4} - \frac{\Pi (2k-\ell, \ell) + \Pi (2k+\ell, \ell)}{4}  \\ &= \frac{\varepsilon (\ell) \varepsilon (|2k-\ell|)+\varepsilon (2k+\ell)}{16} - \mathds{1}_{\{ k = \ell \}} \frac{(2N-k)^+}{16},  \end{split} \]
and the contribution of the terms with $n=m_2$ is the same. The terms with $n \neq m_1, m_2$ contribute
\[ \begin{split} \mathbb{E} \, \frac{1}{2} \sum_{\substack{n,m_1, m_2=1 \\ \textrm{pairwise distinct}}}^N &\cos (2 k \theta_n) \cos (\ell \theta_{m_1}) \cos (\ell \theta_{m_2}) \\ &= \frac{1}{2 \pi^3} \int_0^{\pi} \!\! \int_0^{\pi} \!\! \int_0^{\pi} \det \left( K(x_i, x_j) \right)_{i,j=1}^3 \cos (2k x_1) \cos (\ell x_2) \cos (\ell x_3) \, \mathrm{d}x_1 \mathrm{d} x_2 \mathrm{d} x_3 . \end{split} \]
The $3 \times 3$ determinant in the previous formula has $6$ terms when expanded. Both terms that include $K(x_1, x_1)$ has zero contribution, as $\pi^{-1} \int_0^{\pi} K(x_1, x_1) \cos (2kx_1) \, \mathrm{d} x_1 =\Pi (2k) = 0$. The other two terms corresponding to an odd permutation each contribute
\[ \frac{1}{2\pi^3} \int_0^{\pi} \!\! \int_0^{\pi} \!\! \int_0^{\pi} (-K(x,y)^2 K(z,z)) \cos (2kx) \cos (\ell y) \cos (\ell z) \, \mathrm{d}x \, \mathrm{d}y \, \mathrm{d}z = - \frac{\Pi (\ell) \Pi (2k, \ell)}{2}= - \frac{\varepsilon (2k+\ell)}{8} . \]
The remaining two terms in the determinant each contribute
\begin{multline*}
\frac{1}{2 \pi^3} \int_0^{\pi} \!\! \int_0^{\pi} \!\! \int_0^{\pi} K(x,y) K(y,z) K(z,x) \cos (2kx) \cos (\ell y) \cos (\ell z) \, \mathrm{d}x \, \mathrm{d}y \, \mathrm{d}z \\ = \frac{\Pi (2k, \ell, \ell)}{2} = \mathds{1}_{\{ k = \ell \}} \frac{(N-k)^+}{8} .
\end{multline*}
Therefore
\begin{equation}\label{Fkell}
F(k,\ell) = \frac{\varepsilon (\ell) \varepsilon (|2k-\ell|) - \varepsilon (2k+\ell)}{8} - \mathds{1}_{\{ k = \ell \}} \frac{(2N-k)^+ - 2(N-k)^+}{8} .
\end{equation}

To compute $H(k,\ell)$, we decompose the sum according to the cardinality of the index set $\{ n_1, n_2, m_1, m_2 \}$.\\

\noindent\textbf{Case 1.} Consider the terms in $H(k,\ell)$ for which $\{ n_1, n_2, m_1, m_2 \}$ has cardinality $4$:
\[ \begin{split}
S_4 &\coloneqq \mathbb{E} \sum_{\substack{n_1, n_2, m_1, m_2 =1 \\ \textrm{pairwise distinct}}}^N \cos (k \theta_{n_1}) \cos (k \theta_{n_2}) \cos (\ell \theta_{m_1}) \cos (\ell \theta_{m_2}) \\ &= \frac{1}{\pi^4} \int_0^{\pi} \!\! \int_0^{\pi} \!\! \int_0^{\pi} \!\! \int_0^{\pi} \det \left( K(x_i, x_j) \right)_{i,j=1}^4 \cos (k x_1) \cos (k x_2) \cos (\ell x_3) \cos (\ell x_4) \, \mathrm{d}x_1 \mathrm{d}x_2 \mathrm{d}x_3 \mathrm{d}x_4 . \end{split} \]
The $4 \times 4$ determinant in the previous formula has 24 terms when expanded. Considering each of these terms separately leads to
\[ \begin{split} S_4 &= -4 \Pi (k,k,\ell,\ell) - 2\Pi (k, \ell, k, \ell) + 4 \Pi (k) \Pi (k, \ell, \ell) + 4 \Pi (\ell) \Pi (\ell, k, k) + \Pi (k,k) \Pi (\ell, \ell) \\ &\phantom{={}} + 2\Pi (k,\ell)^2 - \Pi (k)^2 \Pi (\ell, \ell) - \Pi (\ell)^2 \Pi (k,k) - 4 \Pi (k) \Pi (\ell) \Pi (k,\ell) + \Pi (k)^2 \Pi (\ell)^2 . \end{split} \]
By Lemma \ref{gammalemmaSO2N+1}, here
\[ \Pi (k,k) \Pi (\ell, \ell) - \Pi (k)^2 \Pi (\ell, \ell) - \Pi (\ell)^2 \Pi (k,k) + \Pi (k)^2 \Pi (\ell)^2 = \frac{(\varepsilon (k) - (2N-k)^+)(\varepsilon (\ell) - (2N-\ell)^+)}{16} . \]
Hence
\[ S_4 = \frac{(\varepsilon (k) - (2N-k)^+)(\varepsilon (\ell) - (2N-\ell)^+)}{16} + \mathds{1}_{\{ k=\ell \}} S_4' + \mathds{1}_{\{ k \neq \ell \}} S_4'', \]
where
\begin{equation}\label{S4'SO2N+1}
S_4' = -\frac{3(N-k)^+}{2} - \frac{3(2N-k)^+}{8} + \frac{3 \varepsilon (3k) + \varepsilon (k)}{2} + \frac{((2N-k)^+)^2}{8} - \frac{\varepsilon (k) (2N-k)^+}{4}
\end{equation}
and
\begin{equation}\label{S4''SO2N+1}
\begin{split} S_4'' &= - \frac{(N-\max \{ k,\ell \})^+}{2} - (N-k-\ell)^+ - \frac{\alpha (\min \{ k, \ell \}, \max \{ k, \ell \}, \max \{ k, \ell \})}{4} \\ &\phantom{={}} - \frac{\alpha (k+\ell, \ell, k+\ell) + \alpha (k+\ell, k, k+\ell) + \alpha (\ell, \ell-k,\ell) + \alpha (k, k-\ell, k)}{4} \\ &\phantom{={}} - \frac{\alpha (k+\ell, \min \{ k, \ell \}, k+\ell)}{2} - \frac{\alpha (k,k,k+\ell) + \alpha (\ell, \ell, \ell+k)}{4} \\ &\phantom{={}} + \frac{3 \varepsilon (2\ell+k) + \varepsilon (k) \varepsilon (|2\ell -k|)}{4} + \frac{3 \varepsilon (2k+\ell) + \varepsilon (\ell) \varepsilon (|2k -\ell|)}{4} + \frac{\varepsilon (k+\ell)}{2} . \end{split}
\end{equation}

\noindent\textbf{Case 2.} Consider the terms in $H(k,\ell)$ for which $\{ n_1, n_2, m_1, m_2 \}$ has cardinality $3$. Since $n_1 \neq n_2$ and $m_1 \neq m_2$, there are four possibilities: $n_1=m_1$, $n_1=m_2$, $n_2=m_1$ or $n_2=m_2$. All four cases have the same contribution, hence
\[ \begin{split} S_3&\coloneqq \mathbb{E} \sum_{\substack{n_1, n_2, m_1, m_2=1 \\ n_1 \neq n_2, \,\, m_1 \neq m_2 \\ |\{ n_1, n_2, m_1, m_2 \}|=3}}^N \cos (k \theta_{n_1}) \cos (k \theta_{n_2}) \cos (\ell \theta_{m_1}) \cos (\ell \theta_{m_2}) \\ &= \frac{4}{\pi^3} \int_0^{\pi} \!\! \int_0^{\pi} \!\! \int_0^{\pi} \det \left( K(x_i, x_j) \right)_{i,j=1}^3 \frac{\cos ((k-\ell) x_1) + \cos ((k+\ell)x_1)}{2} \cos (kx_2) \cos (\ell x_3) \, \mathrm{d}x_1 \mathrm{d}x_2 \mathrm{d}x_3 . \end{split} \]
The $3 \times 3$ determinant in the previous formula has $6$ terms when expanded. Considering each of these terms separately leads to
\[ \begin{split} S_3 &= 2 \Pi (k-\ell) \Pi (k) \Pi (\ell) +2\Pi (k+\ell) \Pi (k) \Pi (\ell) + 4 \Pi (k-\ell, k, \ell) + 4 \Pi (k+\ell, k, \ell) - 2 \Pi (k-\ell) \Pi (k,\ell) \\ &\phantom{={}} - 2\Pi (k+\ell) \Pi(k, \ell) - 2\Pi (\ell) \Pi (k-\ell, k) - 2 \Pi (\ell) \Pi (k+\ell, k) - 2\Pi (k) \Pi (k-\ell, \ell) - 2\Pi (k) \Pi (k+\ell, \ell) . \end{split} \]
By Lemma \ref{gammalemmaSO2N+1}, we can write $S_3 = \mathds{1}_{\{ k=\ell \}} S_3' + \mathds{1}_{\{ k \neq \ell \}} S_3''$, where
\begin{equation}\label{S3'SO2N+1}
S_3' = \frac{\varepsilon (k) N}{2} + (2N-k)^+ + (N-k)^+ - \frac{N (2N-k)^+}{2} - \varepsilon (k) - \varepsilon (3k)
\end{equation}
and
\begin{equation}\label{S3''SO2N+1}
\begin{split} S_3'' &= (N-\max \{ k,\ell \})^+ + \frac{\alpha (\min \{ k,\ell \}, \max \{ k,\ell \}, \max \{ k,\ell \})}{2} +\frac{\alpha (\ell, \ell-k,\ell) + \alpha (k, k-\ell, k)}{2} \\ &\phantom{={}} +(N-k-\ell)^+ + \frac{\alpha (k+\ell, \ell, k+\ell) + \alpha (k+\ell, k, k+\ell)}{2} \\ &\phantom{={}} - \varepsilon (k+\ell) - \frac{\varepsilon (2 \max \{ k,\ell \} - \min \{ k,\ell \}) + \varepsilon (\max \{ k,\ell \})}{2} - \frac{\varepsilon (2k+\ell) + \varepsilon (2\ell +k)}{2} . \end{split}
\end{equation}

\noindent\textbf{Case 3.} Consider the terms in $H(k,\ell)$ for which $\{ n_1, n_2, m_1, m_2 \}$ has cardinality $2$. Since $n_1 \neq n_2$ and $m_1 \neq m_2$, there are two possibilities: either $n_1=m_1$, $n_2=m_2$, or $n_1=m_2$, $n_2=m_1$. Both cases have the same contribution, hence
\[ \begin{split} S_2 &\coloneqq  \mathbb{E} \sum_{\substack{n_1, n_2, m_1, m_2=1 \\ n_1 \neq n_2, \,\, m_1 \neq m_2 \\ |\{ n_1, n_2, m_1, m_2 \}|=2}}^N \cos (k \theta_{n_1}) \cos (k \theta_{n_2}) \cos (\ell \theta_{m_1}) \cos (\ell \theta_{m_2}) \\ &= \frac{2}{\pi^2} \int_0^{\pi} \!\! \int_0^{\pi} \left( K(x,x) K(y,y) - K(x,y)^2 \right) \cos (kx) \cos (ky) \cos (\ell x) \cos (\ell y) \, \mathrm{d}x \, \mathrm{d} y \\ &= \frac{\Pi (k-\ell)^2 - \Pi (k-\ell, k-\ell)}{2} + \Pi (k-\ell) \Pi (k+\ell) - \Pi (k-\ell, k+\ell) + \frac{\Pi (k+\ell)^2 - \Pi (k+\ell, k+\ell)}{2} . \end{split} \]
By Lemma \ref{gammalemmaSO2N+1}, we can write $S_2 = \mathds{1}_{\{ k=\ell \}} S_2' + \mathds{1}_{\{ k \neq \ell \}} S_2''$, where
\begin{equation}\label{S2'SO2N+1}
S_2' = \frac{N^2-N}{2} - \frac{(N-k)^+}{4}
\end{equation}
and
\begin{equation}\label{S2''SO2N+1}
S_2'' = - \frac{(2N-|k-\ell|)^+}{8} - \frac{(2N-k-\ell)^+}{8} + \frac{\varepsilon (|k-\ell|)}{8} + \frac{3\varepsilon (k+\ell)}{8} .
\end{equation}
Therefore
\[ H(k,\ell) =  \frac{(\varepsilon (k) - (2N-k)^+)(\varepsilon (\ell) - (2N-\ell)^+)}{16} + \mathds{1}_{\{ k=\ell \}} (S_2' + S_3' +S_4') + \mathds{1}_{\{ k \neq \ell \}} (S_2'' + S_3'' + S_4'') , \]
where $S_2'$, $S_3'$, $S_4'$, $S_2''$, $S_3''$, $S_4''$ are given by \eqref{S4'SO2N+1}--\eqref{S2''SO2N+1}.

Substituting \eqref{Ekell}, \eqref{Fkell} and the previous formula in \eqref{secondmomentSO2N+1} yields
\[ \begin{split} \mathbb{E} \, W_2^4 (\mu_A, \lambda_{\mathbb{T}}) &= \frac{4}{(2N)^4} \sum_{k,\ell=1}^{\infty} \frac{(2N-(2N-k)^+ + \varepsilon (k))(2N - (2N-\ell)^+ + \varepsilon (\ell))}{k^2 \ell^2} \\ &\phantom{={}}+\frac{4}{(2N)^4} \sum_{k=1}^{\infty} \frac{T(k)}{k^4} + \frac{4}{(2N)^4} \sum_{\substack{k,\ell =1 \\ k \neq \ell}}^{\infty} \frac{V(k,\ell) + \delta(k,\ell)}{k^2 \ell^2} , \end{split} \]
and as the first line is $(\mathbb{E} \, W_2^2 (\mu_A, \lambda_{\mathbb{T}}))^2$, we obtain
\[ \mathrm{Var} \, W_2^2 (\mu_A, \lambda_{\mathbb{T}}) = \frac{4}{(2N)^4} \sum_{k=1}^{\infty} \frac{T(k)}{k^4} + \frac{4}{(2N)^4} \sum_{\substack{k,\ell =1 \\ k \neq \ell}}^{\infty} \frac{V(k,\ell) + \delta (k,\ell)}{k^2 \ell^2} . \]
Here
\[ \begin{split} T(k) &= 16 \left( E(k,k) + 2 F(k,k) + S_2' + S_3' + S_4' \right) , \\ V(k,\ell) + \delta (k,\ell) &= 16 (F(k,\ell) + F(\ell,k) + S_2'' + S_3'' + S_4'') \end{split} \]
simplify to the expressions given in Theorem \ref{exacttheorem} (ii). This finishes the proof of Theorem \ref{exacttheorem} for $G= \mathrm{SO} (2N+1)$ and $\mathrm{O}(2N+1)$.
\end{proof}

\subsection{Proof for $\mathrm{SO}(2N)$}

\begin{proof}[Proof of Theorem \ref{exacttheorem} (iii)] We rely on the exact formula \eqref{wassersteinothergroupsexpand} for the Wasserstein metric. The points $\theta_n \,\, (1 \le n \le N)$ form a determinantal point process on $[0,\pi]$ with kernel function $K(x,y) =1+2\sum_{j=1}^{N-1} \cos (jx) \cos (jy)$ with respect to the normalized Lebesgue measure on $[0,\pi]$. Let $\Pi (a)$, $\Pi (a,b)$, $\Pi (a,b,c)$ and $\Pi (a,b,c,d)$ be as in \eqref{gammaabcd}, and define $\alpha (a,b,c) = (\min \{ a,N \} + \min \{ b,N \} -c-1)^+$ and $\varepsilon (a) = \mathds{1}_{\{ 1 \le a \le 2N-2, \,\, a \textrm{ even} \}}$.

Following the steps in the proof of Lemma \ref{gammalemmaSO2N+1}, we deduce the following. We have $\Pi (0)=\Pi (0,0)=N$. Let $k,\ell \ge 1$ be integers. We have
\[ \Pi (k) = \frac{\varepsilon (k)}{2}, \qquad \Pi (k,\ell) = \left\{ \begin{array}{ll} \frac{(2N-k-1)^+}{4} + \frac{1}{2} \mathds{1}_{\{ k \le N-1 \}} & \textrm{if } k=\ell, \\ \frac{\varepsilon (k+\ell)}{2} & \textrm{if } k \neq \ell , \end{array} \right. \qquad \Pi (0,k) = \frac{\varepsilon (k)}{2}, \]
and
\[ \begin{split} \Pi (k,\ell, \ell) &= \left\{ \begin{array}{ll} \frac{(2N-k)^+}{8} + \frac{3}{8} \mathds{1}_{\{ k \le N-1 \}} & \textrm{if } k=2\ell, \\ \frac{3 \varepsilon (2\ell +k) + \varepsilon (k) \varepsilon (|2\ell -k|)}{8} & \textrm{if } k \neq 2 \ell, \end{array} \right. \\ \Pi (k,\ell, k+\ell) &= \frac{(N-k-\ell -1)^+}{4} + \frac{\alpha (k+\ell, \ell, k+\ell) + \alpha (k+\ell, k, k+\ell)}{8} \\ &\phantom{={}} + \frac{5}{8} \mathds{1}_{\{ k+\ell \le N-1 \}} + \frac{1}{4} \mathds{1}_{\{ \max \{ k,\ell \} \le N-1 \}}, \end{split} \]
as well as $\Pi (0,k,k)=\frac{(2N-k-1)^+}{4} + \frac{1}{2} \mathds{1}_{\{ k \le N-1 \}}$ and
\[ \Pi (k,k,k,k)= \frac{(2N-k-1)^+}{16} + \frac{(N-k)^+}{4} + \frac{1}{8} \mathds{1}_{\{ k \le N-1 \}} + \frac{1}{4} \mathds{1}_{\{ k \le (N-1)/2 \}} . \]
If $k \neq \ell$, then
\[ \begin{split} \Pi (k,k,\ell,\ell) &=\frac{(N-\max \{ k,\ell \} -1)^+ + (N-k-\ell -1)^+}{8} + \frac{\alpha (\min \{ k,\ell \}, \max \{ k,\ell \}, \max \{ k,\ell \})}{16} \\ &\phantom{={}}+ \frac{\alpha (k+\ell, \ell, k+\ell) + \alpha (k+\ell, k, k+\ell) + \alpha (\ell, \ell-k, \ell) + \alpha (k, k-\ell, k)}{16} +\frac{3}{8} \mathds{1}_{\{ k+\ell \le N-1 \}} \\ &\phantom{={}} + \frac{1}{4} \mathds{1}_{\{ \max \{ k,\ell \} \le N-1 \}} + \frac{1}{8} \mathds{1}_{\{ k \le N-1 \}} \mathds{1}_{\{ |k-\ell| \le N-1 \}} + \frac{1}{8} \mathds{1}_{\{ \ell \le N-1 \}} \mathds{1}_{\{ |k-\ell| \le N-1 \}} , \\ \Pi (k,\ell,k,\ell) &=\frac{(N-k-\ell -1)^+}{4} + \frac{\alpha (k+\ell, \min \{ k,\ell \}, k+\ell)}{4} + \frac{\alpha (k,k,k+\ell) + \alpha (\ell, \ell, k+\ell)}{8} \\ &\phantom{={}} +\frac{1}{2} \mathds{1}_{\{ k+\ell \le N-1 \}} + \frac{1}{2} \mathds{1}_{\{ \max \{ k, \ell \} \le N-1 \}} . \end{split} \]

The rest of the proof is identical to Theorem \ref{exacttheorem} (ii). This finishes the proof of Theorem \ref{exacttheorem} for $G=\mathrm{SO}(2N)$.
\end{proof}

\subsection{Proof for $\mathrm{O}^- (2N+2)$ and $\mathrm{USp}(2N)$}

\begin{proof}[Proof of Theorem \ref{exacttheorem} (iv)] We rely on the exact formula \eqref{wassersteinothergroupsexpand} for the Wasserstein metric. The points $\theta_n \,\, (1 \le n \le N)$ form a determinantal point process on $[0,\pi]$ with kernel function $K(x,y) = 2\sum_{j=1}^N \sin (jx) \sin (jy)$ with respect to the normalized Lebesgue measure on $[0,\pi]$. Let $\Pi (a)$, $\Pi (a,b)$, $\Pi (a,b,c)$ and $\Pi (a,b,c,d)$ be as in \eqref{gammaabcd}, and define $\alpha (a,b,c) = (\min \{ a-1, N \} + \min \{ b-1, N \} - c+1)^+$ and $\varepsilon (a) = \mathds{1}_{\{ 1 \le a \le 2N, \,\, a \textrm{ even} \}}$.

Following the steps in the proof of Lemma \ref{gammalemmaSO2N+1}, we deduce the following. We have $\Pi (0)=\Pi (0,0)=N$. Let $k,\ell \ge 1$ be integers. We have
\[ \Pi (k) = - \frac{\varepsilon (k)}{2}, \qquad \Pi (k,\ell) = \left\{ \begin{array}{ll} \frac{(2N-k+1)^+}{4} - \frac{1}{2}\mathds{1}_{\{ k \le N \}} & \textrm{if } k=\ell, \\ - \frac{\varepsilon (k+\ell)}{2} & \textrm{if } k \neq \ell , \end{array} \right. \qquad \Pi (0,k) = - \frac{\varepsilon (k)}{2}, \]
and
\[ \begin{split} \Pi (k,\ell, \ell) &= \left\{ \begin{array}{ll} \frac{(2N-k)^+}{8} - \frac{3}{8}  \mathds{1}_{\{ k \le N \}} & \textrm{if } k=2\ell, \\ - \frac{3 \varepsilon (2\ell +k) + \varepsilon (k) \varepsilon (|2\ell -k|)}{8} & \textrm{if } k \neq 2 \ell, \end{array} \right. \\ \Pi (k,\ell, k+\ell) &= \frac{(N-k-\ell)^+}{4} + \frac{\alpha (k+\ell, \ell, k+\ell) + \alpha (k+\ell, k, k+\ell)}{8} - \frac{1}{8} \mathds{1}_{\{ k + \ell \le N \}}, \end{split} \]
as well as $\Pi (0,k,k)=\frac{(2N-k+1)^+ }{4} - \frac{1}{2}\mathds{1}_{\{ k \le N \}}$ and
\[ \Pi (k,k,k,k)= \frac{(2N-k+1)^+}{16} + \frac{(N-k)^+}{4} -\frac{1}{8} \mathds{1}_{\{ k \le N \}} - \frac{1}{4} \mathds{1}_{\{ k \le N/2 \}} . \]
If $k \neq \ell$, then
\[ \begin{split} \Pi (k,k,\ell,\ell) &= \frac{(N-\max \{ k,\ell \})^+ + (N-k-\ell)^+}{8} + \frac{\alpha (\min \{ k,\ell \}, \max \{ k,\ell \}, \max \{ k,\ell \})}{16} \\ &\phantom{={}} + \frac{\alpha (k+\ell, \ell, k+\ell) + \alpha (k+\ell, k, k+\ell) + \alpha (\ell, \ell -k, \ell) + \alpha (k, k-\ell, k)}{16} - \frac{1}{8} \mathds{1}_{\{ k+\ell \le N \}} , \\ \Pi (k,\ell,k,\ell) &=\frac{(N-k-\ell)^+}{4} + \frac{\alpha (k+\ell, \min \{ k,\ell \}, k+\ell)}{4} + \frac{\alpha (k,k,k+\ell) + \alpha (\ell, \ell, k+\ell)}{8} . \end{split} \]

The rest of the proof is identical to Theorem \ref{exacttheorem} (ii). This finishes the proof of Theorem \ref{exacttheorem} for $G=\mathrm{O}^-(2N+2)$ and $\mathrm{USp}(2N)$.
\end{proof}

\section{Asymptotic formulas for the expected value and the variance}

In this section, we deduce Theorem \ref{maintheorem} from Theorem \ref{exacttheorem}.

\begin{proof}[Proof of Theorem \ref{maintheorem}] We prove the claim by finding the asymptotics as $N \to \infty$ of the exact formulas for $\mathbb{E} \, W_2^2 (\mu_A, \lambda_{\mathbb{T}})$ and $\mathrm{Var} \, W_2^2 (\mu_A, \lambda_{\mathbb{T}})$ given in Theorem \ref{exacttheorem}. As the proofs are similar for all groups, we only give the details for $\mathrm{SO}(2N+1)$ and $\mathrm{O}(2N+1)$.

Let $G=\mathrm{SO} (2N+1)$ or $\mathrm{O}(2N+1)$. The exact formula for the expected value in Theorem \ref{exacttheorem} yields
\[ \mathbb{E} \, W_2^2 (\mu_A, \lambda_{\mathbb{T}}) = \frac{2}{N_0^2} \sum_{k=1}^{\infty} \frac{\min \{ k,N_0 \} + \mathds{1}_{\{ 1 \le k \le N_0 -1, \,\, k \textrm{ odd} \}}}{k^2} = \frac{2}{N_0^2} \left( \sum_{k=1}^{N_0} \frac{1}{k} + \sum_{k=N_0+1}^{\infty} \frac{N_0}{k^2} + \sum_{\substack{k=1 \\ k \textrm{ odd}}}^{N_0-1} \frac{1}{k^2} \right) . \]
Here
\[ \sum_{k=1}^{N_0} \frac{1}{k} = \log N_0 + \gamma + O \left( \frac{1}{N} \right), \quad \sum_{k=N_0+1}^{\infty} \frac{N_0}{k^2} = 1+O \left( \frac{1}{N} \right), \quad \sum_{\substack{k=1 \\ k \textrm{ odd}}}^{N_0-1} \frac{1}{k^2} = \left( 1-\frac{1}{4} \right) \zeta (2) + O \left( \frac{1}{N} \right) . \]
Therefore
\[  \mathbb{E} \, W_2^2 (\mu_A, \lambda_{\mathbb{T}}) = 2 \frac{\log N_0}{N_0^2} + \frac{2 \gamma + 2+ \frac{\pi^2}{4}}{N_0^2} + O \left( \frac{1}{N^3} \right) , \]
as claimed.

Let $\varepsilon (a)$, $\alpha (a,b,c)$, $T(k)$, $V(k, \ell)$ and $\delta (k,\ell)$ be as in Theorem \ref{exacttheorem} (ii). Observe that
\[ \begin{split} \sum_{k=1}^{\infty} \frac{2 \min \{ k^2, (2N)^2 \}}{k^4} &= \sum_{k=1}^{2N} \frac{2}{k^2} + \sum_{k=2N+1}^{\infty} \frac{8N^2}{k^4} = 2 \zeta (2) + O \left( \frac{1}{N} \right) , \\ \sum_{k=1}^{\infty} \frac{4 \varepsilon (k) \min \{ k,2N \}}{k^4} &= \sum_{\substack{k=1 \\ k \textrm{ odd}}}^{2N-1} \frac{4}{k^3} = 4 \left( 1-\frac{1}{8} \right) \zeta (3) + O \left( \frac{1}{N} \right) . \end{split} \]
We have $\min \{ k,N \} - \min \{ k,2N \} =0$ whenever $1 \le k \le N$, and $O(N)$ otherwise. Similarly, $\varepsilon (3k) - \varepsilon (k) =0$ whenever $1 \le k \le 2N/3$, and $O(1)$ otherwise. Hence
\[ \sum_{k=1}^{\infty} \frac{\min \{ k,N \} - \min \{ k,2N \}}{k^4} = O \left( \frac{1}{N^2} \right) \qquad \textrm{and} \qquad \sum_{k=1}^{\infty} \frac{\varepsilon (3k) - \varepsilon (k)}{k^4} = O \left( \frac{1}{N^3} \right) . \]
The previous four formulas yield
\begin{equation}\label{Tkcontribution}
\frac{4}{N_0^4} \sum_{k=1}^{\infty} \frac{T(k)}{k^4} = \frac{\frac{4 \pi^2}{3} + 14 \zeta (3)}{N_0^4} + O \left( \frac{1}{N^5} \right) .
\end{equation}

One readily checks that $\delta (k,\ell)=0$ whenever $k+\ell \le N$, and of course $\delta (k,\ell) = O(1)$ otherwise. Hence
\begin{equation}\label{deltacontribution}
\sum_{\substack{k,\ell =1 \\ k \neq \ell}}^{\infty} \frac{\delta (k,\ell)}{k^2 \ell^2} = O \left( \frac{1}{N} \right) .
\end{equation}
Finally, note that $V(k,\ell)=V(\ell,k)$, so it is enough to consider the terms $k<\ell$. We can check that $V(k,\ell) =0$ whenever $k+\ell \le 2N$ by considering the four regions
\[ k+\ell \le N \qquad \left. \begin{split} \max \{ k,\ell \} &\le N \\ k+\ell &\ge N \end{split} \right\} \qquad \left. \begin{split} 1 &\le k \le N \\ N &\le \ell \le 2N-k \\ \ell &\le N+k \end{split} \right\} \qquad \left. \begin{split} 1 &\le k \le N/2 \\ N+k &\le \ell \le 2N-k \end{split} \right\} \]
separately, see Figure \ref{fig}. We also have $V(k,\ell)=0$ in the region $\ell \ge 2N+k$.

\begin{figure}[h]
\centering
\begin{tikzpicture}[scale=0.7]

\draw [->] (0,0) -- (6,0) node [below right]  {$k$};
\draw [->] (0,0) -- (0,6) node [above left] {$\ell$};
\draw (6,6) -- (0,0) node [below left] {$0$};
\draw (2,6) -- (0,4) node [left] {$2N$};
\draw (6,2) -- (4,0) node [below] {$2N$};
\draw (0,2) -- (1,3);
\draw (0,2) -- (2,0) node [below] {$N$};
\draw (0,4) -- (4,0);
\draw (2,2) -- (0,2) node [left] {$N$};
\draw (2,0) -- (2,2);
\draw (1,3) -- (1,5);
\fill [lightgray,semitransparent] (0,4) -- (2,6) -- (0,6) -- cycle;
\fill [lightgray,semitransparent] (0,0) -- (0,4) -- (4,0) -- cycle;
\fill [lightgray,semitransparent] (4,0) -- (6,0) -- (6,2) -- cycle;
\end{tikzpicture}
\caption{Regions on which $V(k,\ell)$ is estimated. Note that $V(k,\ell)=0$ in the shaded regions.}
\label{fig}
\end{figure}
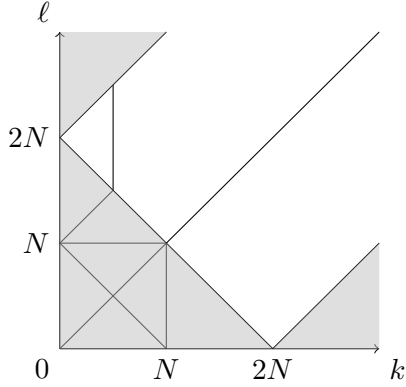

The remaining region is the infinite strip $k+1 \le \ell \le 2N+k$, $k+\ell \ge 2N$. Since $V(k,\ell)=O(N)$, the contribution from $k \ge N/2$ is
\[ \sum_{k \ge N/2} \sum_{\substack{2N-k \le \ell \le 2N+k \\ k+1 \le \ell}} \frac{|V(k,\ell)|}{k^2 \ell^2} \ll \sum_{k \ge N/2} \frac{1}{k^2} = O \left( \frac{1}{N} \right) . \]
In the remaining triangle $1 \le k \le N/2$, $2N-k \le \ell \le 2N+k$ we have $V(k,\ell) = -4N+2\ell -2k + 4 (2N-\ell)^+$, and the contribution is again easily seen to be
\[ \sum_{1 \le k \le N/2} \sum_{2N-k \le \ell \le 2N+k} \frac{|V(k,\ell)|}{k^2 \ell^2} = O \left( \frac{1}{N} \right) . \]
Therefore $\sum_{k,\ell =1, \,\, k \neq \ell}^{\infty} V(k,\ell)/(k^2 \ell^2) = O(1/N)$. Together with \eqref{Tkcontribution} and \eqref{deltacontribution} we finally obtain
\[ \mathrm{Var} \, W_2^2 (\mu_A, \lambda_{\mathbb{T}}) = \frac{4}{N_0^4} \sum_{k=1}^{\infty} \frac{T(k)}{k^4} + \frac{4}{N_0^4} \sum_{\substack{k,\ell =1 \\ k \neq \ell}}^{\infty} \frac{V(k,\ell) + \delta (k,\ell)}{k^2 \ell^2} = \frac{\frac{4 \pi^2}{3} + 14 \zeta (3)}{N_0^4} + O \left( \frac{1}{N^5} \right) , \]
as claimed.
\end{proof}

\section{Limit laws}

Our proof of Theorem \ref{limitlawtheorem} relies on the lengthy computations in Section \ref{exactsection} and the following limit law for the trace of matrix powers.
\begin{thm}[Diaconis--Shahshahani]\label{diaconistheorem} Let $G=\mathrm{U}(N)$, $\mathrm{O}(N)$, $\mathrm{SO}(N)$ or $\mathrm{USp}(2N)$, and let $A \in G$ be a random matrix distributed according to the normalized Haar measure on $G$. For any fixed integer $K \ge 1$,
\[ \left( \mathrm{Tr} (A), \mathrm{Tr} (A^2), \ldots, \mathrm{Tr} (A^K) \right) \overset{d}{\to} (Z_1, Z_2, \ldots, Z_K) \qquad \textrm{as } N \to \infty , \]
where the random variables $Z_k$, $1 \le k \le K$ are defined in terms of i.i.d.\ standard Gaussians $X_k, Y_k$, $1 \le k \le K$ as
\[ Z_k = \left\{ \begin{array}{ll} \sqrt{k/2} (X_k+i Y_k) & \textrm{if } G = \mathrm{U}(N), \\ \sqrt{k} X_k + \mathds{1}_{\{ k \textrm{ even} \}} & \textrm{if } G=\mathrm{O}(N) \textrm{ or } \mathrm{SO}(N), \\ \sqrt{k} X_k - \mathds{1}_{\{ k \textrm{ even} \}} & \textrm{if } G=\mathrm{USp}(2N) . \end{array} \right. \]
\end{thm}

Theorem \ref{diaconistheorem} for $G=\mathrm{U}(N)$, $\mathrm{O}(N)$ and $\mathrm{USp}(2N)$ was first proved by Diaconis and Shahshahani \cite{DI3} by explicitly computing high moments of the traces. Johansson \cite[Corollary 3.5]{JO} gave a new proof based on the Weyl integral formula and a theorem of Szeg\H{o} on Toeplitz determinants. The main ingredients of the latter approach \cite[Proposition 3.1 and Theorem 3.2]{JO} were worked out also for the special orthogonal group, and immediately yield Theorem \ref{diaconistheorem} for $G=\mathrm{SO}(N)$.

\begin{proof}[Proof of Theorem \ref{limitlawtheorem}] Throughout, $X_k, Y_k$, $k \ge 1$ are i.i.d.\ standard Gaussian random variables. Given two real-valued random variables $X$ and $Y$, let
\[ d_L (X,Y) = \inf \left\{ \varepsilon >0 \, : \, \Pr (X \le x- \varepsilon) - \varepsilon \le \Pr (Y \le x) \le \Pr (X \le x+\varepsilon) + \varepsilon \textrm{ for all } x \in \mathbb{R} \right\} \]
denote the L\'evy metric, and recall that it metrizes convergence in distribution.

First, let $G=\mathrm{U}(N)$. Since $\mathrm{Tr} (A^k) = \sum_{n=1}^N e^{i k \theta_n}$, the exact formula \eqref{wassersteinUN} after centering reads
\[ N^2 W_2^2 (\mu_A, \lambda_{\mathbb{T}}) - \mathbb{E} \, N^2 W_2^2 (\mu_A, \lambda_{\mathbb{T}}) = 2 \sum_{k=1}^{\infty} \frac{|\mathrm{Tr} (A^k)|^2 - \mathbb{E} \, |\mathrm{Tr} (A^k)|^2}{k^2} . \]
Fix a large positive integer $K$. Simply replacing $\sum_{k=1}^{\infty}$ by $\sum_{k=K+1}^{\infty}$ throughout the proof of Theorem \ref{exacttheorem} leads to
\[ \mathrm{Var} \Bigg( 2 \sum_{k=K+1}^{\infty} \frac{|\mathrm{Tr} (A^k)|^2}{k^2} \Bigg) = 4 \sum_{k=K+1}^{\infty} \frac{T(k)}{k^4} + 4 \sum_{\substack{k,\ell =K+1 \\ k \neq \ell}}^{\infty} \frac{V(k,\ell) + \delta (k,\ell)}{k^2 \ell^2} \ll \frac{1}{K} . \]
An application of the Chebyshev inequality thus shows that
\[ \Pr \left( \left| \sum_{k=K+1}^{\infty} \frac{|\mathrm{Tr} (A^k)|^2 - \mathbb{E} \, |\mathrm{Tr} (A^k)|^2}{k^2} \right| \ge \frac{1}{K^{1/3}} \right) \ll \frac{1}{K^{1/3}} . \]
Since $\mathrm{Var} \, (\sum_{k=K+1}^{\infty} (X_k^2 + Y_k^2 -2) /k) = \sum_{k=K+1}^{\infty} (\mathrm{Var} (X_k^2 + Y_k^2))/k^2 \ll 1/K$, the Chebyshev inequality similarly shows that
\[ \Pr \left( \left| \sum_{k=K+1}^{\infty} \frac{X_k^2+Y_k^2 -2}{k} \right| \ge \frac{1}{K^{1/3}} \right) \ll \frac{1}{K^{1/3}} . \]
The previous two formulas immediately yield
\begin{equation}\label{dLupperbound}
\begin{split} d_L \bigg( N^2 W_2^2 (\mu_A, \lambda_{\mathbb{T}}) - &\mathbb{E} \, N^2 W_2^2 (\mu_A, \lambda_{\mathbb{T}}) , \sum_{k=1}^{\infty} \frac{X_k^2 + Y_k^2 -2}{k} \bigg) \\ &\le d_L \bigg( 2 \sum_{k=1}^K \frac{|\mathrm{Tr} (A^k)|^2 - \mathbb{E} \, |\mathrm{Tr} (A^k)|^2}{k^2} , \sum_{k=1}^K \frac{X_k^2 + Y_k^2 -2}{k} \bigg) + O \bigg( \frac{1}{K^{1/3}} \bigg) . \end{split}
\end{equation}
An application of Theorem \ref{diaconistheorem} gives
\[ 2 \sum_{k=1}^K \frac{|\mathrm{Tr} (A^k)|^2}{k^2} \overset{d}{\to} \sum_{k=1}^K \frac{X_k^2+Y_k^2}{k} \qquad \textrm{as } N \to \infty . \]
As observed in the proof of Theorem \ref{maintheorem}, $\mathbb{E} |\sum_{n=1}^N e^{i k \theta_n}|^2 = \min \{ k,N \}$. Hence for all $N \ge K$,
\[ 2 \sum_{k=1}^K \frac{\mathbb{E} \, |\mathrm{Tr} (A^k)|^2}{k^2} = \sum_{k=1}^K \frac{2}{k} . \]
Taking the limit as $N \to \infty$ in \eqref{dLupperbound} thus leads to
\[ \limsup_{N \to \infty} d_L \left( N^2 W_2^2 (\mu_A, \lambda_{\mathbb{T}}) - \mathbb{E} \, N^2 W_2^2 (\mu_A, \lambda_{\mathbb{T}}) , \sum_{k=1}^{\infty} \frac{X_k^2 + Y_k^2 -2}{k} \right) = O \left( \frac{1}{K^{1/3}} \right) . \]
Finally, taking the limit as $K \to \infty$ yields
\[ N^2 W_2^2 (\mu_A, \lambda_{\mathbb{T}}) - \mathbb{E} \, N^2 W_2^2 (\mu_A, \lambda_{\mathbb{T}}) \overset{d}{\to} \sum_{k=1}^{\infty} \frac{X_k^2 + Y_k^2 -2}{k} . \]
By Theorem \ref{maintheorem}, the centering term $\mathbb{E} \, N^2 W_2^2 (\mu_A, \lambda_{\mathbb{T}})$ can be replaced by $2 \log N + c_{\mathrm{U}(N)}$. This finishes the proof of Theorem \ref{limitlawtheorem} for $G=\mathrm{U}(N)$. Lemma \ref{reductionlemma} shows that Theorem \ref{limitlawtheorem} holds also for $G=\mathrm{SU}(N)$.

The proof for the other groups is entirely analogous, therefore we only point out the necessary modifications. Let $G=\mathrm{SO}(2N+1)$. Taking the trivial eigenvalue $1$ into account, we have $\mathrm{Tr} (A^k) = 1+2 \sum_{n=1}^N \cos (k \theta_n)$. The exact formula \eqref{wassersteinothergroups} after centering thus reads
\[ N_0^2 W_2^2 (\mu_A, \lambda_{\mathbb{T}}) - \mathbb{E} \, N_0^2 W_2^2 (\mu_A, \lambda_{\mathbb{T}}) = 2 \sum_{k=1}^{\infty} \frac{(\mathrm{Tr} (A^k) -1)^2 - \mathbb{E} \, (\mathrm{Tr} (A^k) -1)^2}{k^2} . \]
Theorem \ref{diaconistheorem} now gives
\[ 2 \sum_{k=1}^K \frac{(\mathrm{Tr} (A^k) -1)^2}{k^2} \overset{d}{\to} 2 \sum_{k=1}^K \frac{(\sqrt{k} X_k - \mathds{1}_{\{ k \textrm{ odd} \}})^2}{k^2} , \]
which eventually leads to
\[ N_0^2 W_2^2 (\mu_A, \lambda_{\mathbb{T}}) - \mathbb{E} \, N_0^2 W_2^2 (\mu_A, \lambda_{\mathbb{T}}) \overset{d}{\to} 2 \sum_{k=1}^{\infty} \frac{X_k^2 -  \mathds{1}_{\{ k \textrm{ odd}\}} (2/\sqrt{k}) X_k -1}{k} . \]
This proves Theorem \ref{limitlawtheorem} for $G=\mathrm{SO} (2N+1)$. The same holds for $G=\mathrm{O}(2N+1)$ by Lemma \ref{reductionlemma}.

In the case of $G=\mathrm{SO}(2N)$ and $\mathrm{USp}(2N)$, there are no trivial eigenvalues. We have $\mathrm{Tr} (A^k) = 2 \sum_{n=1}^N \cos (k \theta_n)$, and an application of Theorem \ref{diaconistheorem} leads to the desired limit laws. Finally, note that $W_2(\mu_A, \lambda_{\mathbb{T}})$ has the same distribution for $G=\mathrm{O}^- (2N+2)$ and $G=\mathrm{USp}(2N)$, as the nontrivial eigenvalues are determinantal point processes with the same kernel, cf.\ Table \ref{tabledeterminantal}.
\end{proof}

\section*{Appendix}

Let $X_k, Y_k$, $k \ge 1$ be i.i.d.\ standard Gaussians, and let
\[ \xi_G = \left\{ \begin{array}{ll} \displaystyle{\sum_{k=1}^{\infty} \frac{X_k^2 + Y_k^2 -2}{k}} & \textrm{if } G=\mathrm{U}(N) \textrm{ or } \mathrm{SU}(N), \\ \displaystyle{2 \sum_{k=1}^{\infty} \frac{X_k^2 -\mathds{1}_{\{ k \textrm{ odd} \}} (2/\sqrt{k}) X_k-1}{k}} & \textrm{if } G=\mathrm{SO}(2N+1) \textrm{ or } \mathrm{O}(2N+1), \\ \displaystyle{2 \sum_{k=1}^{\infty} \frac{X_k^2 -\mathds{1}_{\{ k \textrm{ even} \}} (2/\sqrt{k}) X_k-1}{k}} & \textrm{if } G=\mathrm{SO}(2N), \mathrm{O}^- (2N+2) \textrm{ or } \mathrm{USp} (2N) \end{array} \right. \]
be the limiting random variable in Theorem \ref{limitlawtheorem}, as defined in Table \ref{tablexiG}. In all three cases, we have a series of independent, centered random variables whose variances decay at the rate $\ll k^{-2}$. In particular, all three series converge a.s.\ and in $L^2$. Clearly, $\mathbb{E} \, \xi_G =0$, and by independence and convergence in $L^2$, the variances are
\[ \mathrm{Var} \Bigg( \sum_{k=1}^{\infty} \frac{X_k^2 + Y_k^2 -2}{k} \Bigg) = \sum_{k=1}^{\infty} \frac{\mathrm{Var} (X_k^2 + Y_k^2 -2)}{k^2} = \sum_{k=1}^{\infty} \frac{4}{k^2} = \frac{2 \pi^2}{3}, \]
and
\[ \begin{split} \mathrm{Var} \Bigg( 2 \sum_{k=1}^{\infty} &\frac{X_k^2 -\mathds{1}_{\{ k \textrm{ odd} \}} (2/\sqrt{k}) X_k-1}{k} \Bigg) \\ &= 4 \sum_{k=1}^{\infty} \frac{\mathrm{Var} (X_k^2 -\mathds{1}_{\{ k \textrm{ odd} \}} (2/\sqrt{k}) X_k-1)}{k^2} = 4 \sum_{k=1}^{\infty} \frac{2+ \mathds{1}_{\{ k \textrm{ odd} \}}(4/k)}{k^2} = \frac{4 \pi^2}{3} + 14 \zeta (3), \\ \mathrm{Var} \Bigg( 2 \sum_{k=1}^{\infty} &\frac{X_k^2 -\mathds{1}_{\{ k \textrm{ even} \}} (2/\sqrt{k}) X_k-1}{k} \Bigg)  \\ &= 4 \sum_{k=1}^{\infty} \frac{\mathrm{Var} (X_k^2 -\mathds{1}_{\{ k \textrm{ even} \}} (2/\sqrt{k}) X_k-1)}{k^2} =4 \sum_{k=1}^{\infty} \frac{2+ \mathds{1}_{\{ k \textrm{ even} \}}(4/k)}{k^2} = \frac{4 \pi^2}{3} + 2 \zeta (3) . \end{split} \]
This proves that $\mathrm{Var} \, \xi_G = \sigma_G$ for all $G$.

We now find the characteristic function of $\xi_G$. Consider first $G=\mathrm{U}(N)$ and $\mathrm{SU}(N)$. Recall that $X_k^2 + Y_k^2$ has chi-squared distribution with $2$ degrees of freedom, and its characteristic function is $(1-2it)^{-1}$. By independence, the characteristic function of $\xi_G$ is
\[ \mathbb{E} \, e^{it \xi_G} = \prod_{k=1}^{\infty} \mathbb{E} \, e^{it(X_k^2 + Y_k^2 -2)/k} = \prod_{k=1}^{\infty} \left( 1 - \frac{2it}{k} \right)^{-1} e^{-2it/k} = \Gamma (1-2it) e^{-2\gamma it} . \]
The infinite product in the previous formula appears in the Weierstrass product form of the gamma function, $\Gamma (z+1) e^{\gamma z} = \prod_{k=1}^{\infty} (1+z/k)^{-1} e^{z/k}$, valid for $z \in \mathbb{C} \backslash \{ -1, -2, \ldots \}$ \cite[p.\ 3]{AAR}.

Next, consider $G=\mathrm{SO} (2N+1)$ and $\mathrm{O}(2N+1)$. Recall that for any constant $c \in \mathbb{R}$, $(X_k+c)^2$ has noncentral chi-squared distribution with $1$ degree of freedom, and its characteristic function is $(1-2it)^{-1/2} e^{ic^2 t / (1-2it)}$. Therefore
\[ \begin{split} \mathbb{E} \, \exp \left( it 2\frac{X_k^2-1}{k} \right) &= \left( 1-\frac{4it}{k} \right)^{-1/2} e^{- 2it/k}, \\ \mathbb{E} \, \exp \left( it 2 \frac{ (X_k-1/\sqrt{k})^2 - 1/k-1}{k} \right) &= \left( 1-\frac{4it}{k} \right)^{-1/2} e^{- 2it/k} e^{2it/(k^2-4itk) - 2it/k^2} . \end{split} \]
By independence, the characteristic function of $\xi_G$ is
\[ \begin{split} \mathbb{E} \, e^{it \xi_G} &= \prod_{k=1}^{\infty} \left( 1 -\frac{4it}{k} \right)^{-1/2} e^{-2it/k} \prod_{\substack{k=1 \\ k \textrm{ odd}}}^{\infty} e^{2it/(k^2-4itk) - 2it/k^2} \\ &= \Gamma (1-4it)^{1/2} e^{-2\gamma it} \exp \left( \sum_{\substack{k=1 \\ k \textrm{ odd}}}^{\infty} \frac{2it}{k^2 -4itk} - \sum_{\substack{k=1 \\ k \textrm{ odd}}}^{\infty} \frac{2it}{k^2} \right) . \end{split} \]
The series representation of the digamma function $\psi (z+1)+\gamma = \sum_{k=1}^{\infty} z/(k^2+kz)$, valid for $z \in \mathbb{C} \backslash \{ -1, -2, \ldots \}$ \cite[p.\ 13]{AAR}, shows that
\[ \sum_{k=1}^{\infty} \frac{2it}{k^2-4itk} = - \frac{\psi (1-4it)+\gamma}{2} \qquad \textrm{and} \qquad \sum_{\substack{k=1 \\ k \textrm{ even}}}^{\infty} \frac{2it}{k^2-4itk} = - \frac{\psi (1-2it) + \gamma}{4} . \]
Therefore
\[ \mathbb{E} \, e^{it \xi_G} = \Gamma (1-4it)^{1/2} \exp \left( - \left( 2\gamma + \frac{\pi^2}{4} \right) it - \frac{2 \psi (1-4it) - \psi (1-2it)+\gamma}{4} \right) . \]

In the case of $G=\mathrm{SO}(2N)$, $\mathrm{O}^- (2N+2)$ and $\mathrm{USp}(2N)$, we similarly deduce that the characteristic function of $\xi_G$ is
\[ \begin{split} \mathbb{E} \, e^{it \xi_G} &= \Gamma (1-4it)^{1/2} e^{-2\gamma it} \exp \left( \sum_{\substack{k=1 \\ k \textrm{ even}}}^{\infty} \frac{2it}{k^2 -4itk} - \sum_{\substack{k=1 \\ k \textrm{ even}}}^{\infty} \frac{2it}{k^2} \right) \\ &= \Gamma (1-4it)^{1/2} \exp \left( - \left( 2\gamma + \frac{\pi^2}{12} \right) it - \frac{\psi (1-2it)+\gamma}{4} \right) . \end{split} \]
This concludes the computation of the characteristic function of $\xi_G$ for all $G$, as given in Table \ref{tablexiGchar}.

Finally, observe that the substitution $y=e^{-x/2}$ in the integral representation of the gamma function yields
\[ \Gamma (1-2it) = \int_0^{\infty} y^{-2it} e^{-y} \, \mathrm{d} y = \int_{-\infty}^{\infty} e^{itx} \frac{1}{2} e^{-x/2-e^{-x/2}} \, \mathrm{d}x . \]
In particular, $\Gamma (1-2it) e^{-2\gamma it}$ is the Fourier transform of the function
\[ f(x) = \frac{1}{2} \exp \left( -\frac{x+2 \gamma}{2} - \exp \left( - \frac{x+2 \gamma}{2} \right) \right) , \qquad x \in \mathbb{R} . \]
This proves that the density function of $\xi_{\mathrm{U}(N)} = \xi_{\mathrm{SU}(N)}$ is $f(x)$.

\section*{Acknowledgments} The author is supported by the Austrian Science Fund (FWF) project M 3260-N.

\end{document}